\documentclass[12pt,letterpaper,bibliography=totoc]{scrartcl}

\usepackage{amsmath}
\usepackage{amsthm}
\usepackage{amssymb}
\usepackage{mathtools}
\usepackage[mathscr]{euscript}
\usepackage{xcolor}
\definecolor{Myblue}{rgb}{0,0,0.6}
\usepackage[
 colorlinks,
 citecolor=Myblue,
 linkcolor=Myblue,
 urlcolor=Myblue,
 pdfpagemode=UseNone
]{hyperref}
\hypersetup{
 pdftitle={Asymptotics of the Moduli Spaces of Parabolic SL(3,C)-Higgs Bundles},
 pdfauthor={Shiyu Cao}
}
\usepackage[totalwidth=420pt,totalheight=590pt]{geometry}
\usepackage[inline,shortlabels]{enumitem}
\DeclareTOCStyleEntry[pagenumberwidth=2.5em]{tocline}{section}

\DeclareOldFontCommand{\rm}{\normalfont\rmfamily}{\mathrm}

\allowdisplaybreaks

\newcommand{\C}{\mathbb C}
\newcommand{\R}{\mathbb R}
\newcommand{\Z}{\mathbb Z}
\newcommand{\Pp}{\mathbb P}
\newcommand{\calM}{\mathcal M}
\newcommand{\calL}{\mathcal L}
\newcommand{\calV}{\mathcal V}
\newcommand{\Id}{\operatorname{Id}}
\newcommand{\SL}{\operatorname{SL}}

\newcommand{\ParEnd}{\operatorname{ParEnd}}
\newcommand{\SParEnd}{\operatorname{SParEnd}}
\newcommand{\Pic}{\operatorname{Pic}}
\newcommand{\Nm}{\operatorname{Nm}}
\newcommand{\tr}{\operatorname{tr}}
\newcommand{\diag}{\operatorname{diag}}
\newcommand{\pardeg}{\operatorname{pardeg}}
\newcommand{\rank}{\operatorname{rank}}
\newcommand{\Area}{\operatorname{Area}}
\newcommand{\inj}{\operatorname{inj}}
\newcommand{\Vol}{\operatorname{Vol}}
\newcommand{\Rm}{\operatorname{Rm}}
\newcommand{\ad}{\operatorname{ad}}
\newcommand{\dd}{\,d}
\newcommand{\ii}{\mathrm i}

\theoremstyle{plain}
\newtheorem{thm}{Theorem}
\newtheorem{prop}[thm]{Proposition}
\newtheorem{lem}[thm]{Lemma}
\theoremstyle{definition}
\newtheorem{defn}[thm]{Definition}

\numberwithin{equation}{section}
\numberwithin{thm}{section}

\title{Asymptotics of the Moduli Spaces of Parabolic $\mathrm{SL}(3,\mathbb C)$-Higgs Bundles}
\author{Shiyu Cao\\[-1mm]
  \normalsize Chern Institute of Mathematics, Nankai University\\
  \normalsize\href{mailto:shiyucao@126.com}{shiyucao@126.com}}
\date{October 2026}

\begin{document}
\setlength{\emergencystretch}{1.5em}

\maketitle

\begin{abstract}
We study the large-scale geometry of the moduli spaces of stable
strongly parabolic $\mathrm{SL}(3,\C)$-Higgs bundles on the
three-punctured sphere for a family of parabolic weights.
We glue the local Toda model metric to a decoupled metric obtained from the
spectral line bundle, and perturb the resulting approximate metrics
to harmonic metrics adapted to the parabolic structure.  We also prove
that Hitchin's $L^2$ metric approaches the semi-flat metric
exponentially.  The metric comparison and the spectral period calculation show that each
moduli space is a complete ALG gravitational instanton of type $E_6$.
The asymptotic estimates are uniform on compact subintervals of the
weight chamber.
\end{abstract}

\thispagestyle{empty}

\setcounter{tocdepth}{2}
\tableofcontents

\section{Introduction}\label{sec:introduction}

A Higgs bundle on a compact Riemann surface consists of a holomorphic
vector bundle and a holomorphic endomorphism-valued one-form.  A
Hermitian metric solving Hitchin's equations identifies its stable
moduli space with a gauge-theoretic moduli space carrying a natural
hyperk\"ahler metric \cite{Hit87}.  On a punctured curve, parabolic
flags and weights prescribe the growth of the metric, and the
correspondence extends to stable parabolic Higgs bundles
\cite{Simpson1990,Konno1993,Mochizuki}.

The characteristic polynomial defines the Hitchin fibration.  Over
its regular locus, the spectral correspondence identifies the fibers
with abelian varieties \cite{BNR89,Markman94}.  The special K\"ahler
metric on the base and the flat metrics on these fibers define the
semi-flat metric; its horizontal and vertical parts are spectral
Hodge pairings \cite[Proposition~2.14 and Theorem~2.15]
{Fredrickson-1810.01554}.  The asymptotic picture of
Gaiotto--Moore--Neitzke predicts exponentially small corrections to
this metric in the large-Higgs-field limit \cite{GMN13}.

Mazzeo--Swoboda--Weiss--Witt constructed rank-two limiting
configurations and their desingularizations \cite{MSWW16}, and
identified the semi-flat leading term in the Hitchin metric
\cite{MSWW19}.  Exponential comparison was proved on the Hitchin
section by Dumas--Neitzke \cite{DN19}, and by Fredrickson on the
higher-rank locus where all ramification indices of the spectral
cover are at most two
\cite[Theorem~4.3]{Fredrickson-1810.01554}.  Mochizuki established
the comparison on the smooth spectral locus in arbitrary rank,
including higher ramification \cite{Mochizuki23}.
For punctured curves, Fredrickson--Mazzeo--Swoboda--Weiss treated
parabolic $\SL(2,\C)$-Higgs bundles \cite{FMSW22}, and Chen--Li
studied rank-two irregular Higgs bundles \cite{ChenLi}.  Their
four-dimensional examples relate these estimates to ALG geometry.

A hyperk\"ahler gravitational instanton is a complete noncompact
hyperk\"ahler four-manifold with square-integrable curvature.
An ALG end is asymptotic, with all derivatives, to a flat torus
fibration with finite monodromy over the complement of a disk in a
two-dimensional Euclidean cone.
It has quadratic volume growth, and its flat model is specified by
the cone angle, elliptic modulus, and fiber scale.  Chen--Chen proved
that, under the curvature bound $|\Rm|=O(r^{-2-\epsilon})$ for some
$\epsilon>0$, hyperk\"ahler gravitational instantons have ALE, ALF,
ALG, or ALH ends \cite{CC21a}, where $r$ is the distance from a
fixed point.

Chen--Chen also showed that every nonflat ALG gravitational instanton
admits, for a suitable compatible complex structure, a compactification
to a rational elliptic surface by adding a singular fiber with finite
monodromy \cite{CC21a,CC21b}.  The Kodaira type of this fiber determines
the asymptotic cone angle.  These results connect the metric
asymptotics of four-dimensional Hitchin systems with their elliptic
compactifications.  Li--Sun \cite[Sections~4.2.3 and~4.5.1]{LiSun26}
survey the asymptotic models and compactification theory of gravitational
instantons.

We study strongly parabolic rank-three Higgs bundles on
$$
 X=\Pp^1,\qquad D=\{0,1,\infty\}.
$$
The underlying bundle has degree $-3$ and full flags, with weights
$(0,1/3,2/3+\delta)$ at $0$ and $(0,1/3,2/3)$ at $1,\infty$.
We use the degree-zero determinant-one normalization in
Subsection~\ref{subsec:moduli}, and denote the resulting moduli
space by $\calM_\delta$.  Proposition~\ref{prop:chamber} shows that
$0<\delta<1/3$ is one stability chamber.  All uniform estimates
below hold on an arbitrary compact subinterval $I\Subset(0,1/3)$.

Throughout, the Hitchin fibration is denoted by $h:\calM_\delta\to\C_t$.
For $t\ne0$, the corresponding spectral curve is a smooth elliptic triple cover
of $X$, totally ramified over $D$.  Let $Y$ be the spectral curve for
$t=1$.  After the base change $t=\lambda^3$,
the regular locus is parametrized by $(\lambda,q)\in\C^\times\times\Pic^0(Y)$,
with quotient action
$$
 (\lambda,q)\longmapsto(\omega\lambda,\omega^{-1}q),
 \qquad \omega=e^{2\pi\ii/3}.
$$
The action on $\Pic^0(Y)$ is pullback by the deck generator, expressed
as multiplication by $\omega^{-1}$ through the Abel--Jacobi isomorphism.
We put $T=|\lambda|=|t|^{1/3}$.  The spectral description is
compatible with the orbifold construction of Groechenig
\cite{Gro12}; the rank-three $\widetilde E_6$ family was also
studied by Eper--Szab\'o \cite{ES26}, in which they proved the geometric $P=W$ conjecture.

Our first result constructs the harmonic metric uniformly on compact
subintervals of the chamber.  The metric on the spectral line bundle gives a decoupled metric
$h_q^{\rm dec}$.  We interpolate it with the local cyclic Toda
metrics to obtain $h_{\lambda,q,\delta}^{\rm app}$, as in
Definition~\ref{defn:approximate-metric}.

\begin{thm}\label{thm:intro-construction}
For every $I\Subset(0,1/3)$ there are $C_I,c_I,T_I>0$ such that
the normalized adapted harmonic metric satisfies
\begin{equation}\label{eq:intro-harmonic-approximation}
 \begin{aligned}
 h_{\lambda,q,\delta}
 &=h_{\lambda,q,\delta}^{\rm app}e^{2\gamma_{\lambda,q,\delta}},\\
 \|\gamma_{\lambda,q,\delta}\|_{L^\infty,h^{\rm app}}
 &\leq C_Ie^{-c_IT},\qquad T\geq T_I,
 \end{aligned}
\end{equation}
uniformly for $(q,\delta,\arg\lambda)\in \Pic^0(Y)\times I\times S^1$.
In local trivializations of the underlying smooth family, the metric
depends smoothly on $(\lambda,q)$ for fixed $\delta$ and continuously on
$\delta$ in $C^\infty_{\mathrm{loc}}(X\setminus D)$.  It is compatible
with the quotient action.
\end{thm}

The local solutions are obtained from the cyclic Toda equation
\cite{FN17,LM2010b}.  The uniform endpoint estimates make the
residual in Hitchin's equation for $h^{\rm app}$ exponentially small
on a fixed transition annulus.  We prove a uniform coercive estimate
for the linearized equation and a polynomial inverse bound between the weighted spaces
specified in Definition~\ref{defn:holder-domain} and
Theorem~\ref{thm:uniform-inverse}, following
\cite[Section~5.2]{FMSW22}.  The contraction also bounds
$\gamma$ exponentially in a weighted $C^{2,\alpha}$ norm, including
its constant diagonal coefficients at the punctures; see
Proposition~\ref{prop:correction}.

Our second result compares the Hitchin metric with the semi-flat
metric on all tangent directions.

\begin{thm}\label{thm:intro-metric-comparison}
For every $I\Subset(0,1/3)$ there are $C_I,c_I,T_I>0$ such that
\begin{equation}\label{eq:intro-metric-decay}
 \sup_{\substack{\delta\in I\\(\phi,q)\in S^1\times\Pic^0(Y)}}
 |g_{L^2,\delta}-g_{{\rm sf},\delta}|_{g_{{\rm sf},\delta}}
 \leq C_Ie^{-c_IT},\qquad T\geq T_I,
\end{equation}
after pullback by the spectral parametrization
$\widetilde\Psi_\delta$ of Proposition~\ref{prop:spectral}.
\end{thm}

We construct horizontal and vertical representatives from the
spectral data and compare their pairings, following the approach
of \cite{Mochizuki23}.  Near the punctures, the Toda equation and
Green's formula reduce the horizontal and vertical pairing errors
to boundary terms.  The cyclic matrix structure makes the
approximate mixed pairing vanish.  The parabolic exponents ensure that the inner
boundary terms vanish uniformly on $I$.  Uniform coercivity
controls the correction to harmonic representatives.  We then
identify their deformation classes and transfer the pairing
estimates to $g_{L^2}$.  This proves Theorem~\ref{thm:metric-comparison}.

Finally, we show that the spectral periods determine a flat ALG model.  We
use the notation $\mathcal C_{\beta,\tau,L}$ of
\cite[Definitions~1.1 and~2.5]{CV24}, with metric $g_{\mathcal C}$,
radial coordinate $R_{\rm ALG}$, and cover coordinates $(u,v)$.
Write $\Omega_\delta^{\rm Serre}$ for the holomorphic symplectic form
defined by the trace and Serre duality, and set
$\Omega_\delta=\omega_2+\ii\omega_3$.
For deformation pairs ordered as $(a,\varphi)$ we have
\begin{equation}\label{eq:holomorphic-symplectic-pairing}
 \Omega_\delta((a_1,\varphi_1),(a_2,\varphi_2))
 =2\int_X\tr(\varphi_1\wedge a_2-\varphi_2\wedge a_1).
\end{equation}

\begin{thm}\label{thm:main}
For $0<\delta<1/3$, the Hitchin metric on $\calM_\delta$ is a
complete ALG gravitational instanton of type $E_6$, with parameters
\begin{equation}\label{eq:intro-alg-parameters}
 (\beta,\tau,L)=
 \left(\frac13,e^{2\pi\ii/3},
             2\pi\sqrt{\frac2{\sqrt3}}\right).
\end{equation}
There is a biholomorphic end identification
$$
 \Psi_\delta:
 (\{|\lambda|>R\}\times\Pic^0(Y))/\mu_3
 \longrightarrow h^{-1}\{|t|>R^3\}
$$
such that
$$
 \Psi_\delta^*\Omega_\delta=du\wedge dv,
 \qquad
 \Psi_\delta^*\Omega_\delta^{\rm Serre}
 =\frac1{4\pi\ii}\,du\wedge dv.
$$
For each $I\Subset(0,1/3)$, the radius $R$ and a constant $c_I>0$
can be chosen so that, for every integer $k\geq0$,
\begin{equation}\label{eq:intro-derivative-decay}
 \sup_{\delta\in I}
 \left|(\nabla^{g_{\mathcal C}})^k
       (\Psi_\delta^*g_{L^2,\delta}-g_{\mathcal C})
       \right|_{g_{\mathcal C}}
 \leq C_{I,k}e^{-c_IR_{\rm ALG}}.
\end{equation}
The analogous estimates hold for the three K\"ahler forms.
In particular, the ALG order is at least $2$.
\end{thm}

The period calculation fixes the lattice, quotient action, and metric
constants independently of $\delta$.  The estimates for derivatives
of every order follow from Theorem~\ref{thm:intro-metric-comparison}, local
Ricci-flat K\"ahler regularity \cite[Proposition~2.1]{HT20}, and
interpolation.  The flat end also gives quadratic volume growth:
for a base point $p$ there are $c,C,\rho_0>0$ with
$$
 c\rho^2\leq\Vol_{g_{L^2}}B(p,\rho)\leq C\rho^2,
 \qquad\rho\geq\rho_0.
$$

The complex surface and its elliptic compactification are independent
of $\delta$ in the chamber.  The isomorphism induced by the
Fourier--Mukai construction in \cite[Theorem~4.1]{Gro12} preserves
the holomorphic symplectic forms by the calculation in
\cite[Lemma~3.2 and the proof of Proposition~3.3]{Jia26}.
The resolution of $(Y\times\Pp^1)/\mu_3$ described in
\cite[Proposition~6.3 and Lemma~6.10]{Huang26}, followed by contraction
of the boundary $(-1)$-curve, gives a relatively minimal elliptic
compactification.  Its only singular fibers have Kodaira types $IV^*$
over zero and $IV$ at infinity.  The type-$IV$ fiber at infinity corresponds to the cone angle
$2\pi/3$ of the ALG model \cite{CC21b}.  The dual intersection graph of the central type-$IV^*$ fiber is the
affine $E_6$ diagram.
Theorem~\ref{thm:elliptic-compactification} describes the boundary
divisor and its relatively minimal image.

The paper is organized as follows.  Section~\ref{sec:preliminaries} fixes the parabolic and spectral
data.  Section~\ref{sec:local-analysis} constructs the Toda
models and proves their uniform estimates.  Section
\ref{sec:global-analysis} constructs the approximate metric and
solves the nonlinear equation.  Section~\ref{sec:metric-comparison}
proves the exponential metric comparison, and
Section~\ref{sec:alg} identifies the compactification and the ALG
geometry.

\subsection*{Acknowledgments}

The author thanks Prof.~Qiongling Li for her support and encouragement,
and Prof.~Takuro Mochizuki for his mini-courses at USTC, Hefei in the autumn
of 2024 and for his generous explanations about his seminal works in this area.

\subsection*{Disclosure on the Use of AI}
The research idea is the author's.  During the preparation of this work,
the author used ChatGPT for literature searches, checking calculations,
proofreading, and language editing.  Notably, ChatGPT brought \cite{HT20} to the
author's attention; its local regularity estimate is used in
Theorem~\ref{thm:derivative-decay}.  It also suggested \cite{Huang26,Jia26},
which are used in Section~\ref{sec:alg}.  The author takes full
responsibility for the contents of this paper.

\section{Preliminaries}
\label{sec:preliminaries}
In this section, we introduce parabolic Higgs bundles, their moduli spaces,
and the Hitchin and semi-flat hyperk\"ahler metrics.
\subsection{Parabolic Higgs bundles}

\begin{defn}
Let $X$ be a smooth projective curve and $D\subset X$ a finite set.
A parabolic rank-three bundle with full flags consists of a holomorphic
bundle $E$, flags
$$
 E_p=F_p^1\supset F_p^2\supset F_p^3\supset F_p^4=0,
 \qquad p\in D,
$$
with one-dimensional successive quotients, and real lifts of the weights
$$
 \alpha_{p,1}<\alpha_{p,2}<\alpha_{p,3}<\alpha_{p,1}+1.
$$
A logarithmic Higgs field is a holomorphic map
$$
 \Phi:E\longrightarrow E\otimes K_X(D)
$$
whose residues preserve the flags.
\end{defn}

\begin{defn}
The Higgs field is called strongly parabolic if
\begin{equation}\label{eq:strongly-parabolic-condition}
 \operatorname{Res}_p\Phi(F_p^j)\subset F_p^{j+1}
 \qquad (p\in D,\ 1\leq j\leq3).
\end{equation}

The parabolic degree and slope are defined as
\begin{equation}\label{eq:parabolic-degree}
 \pardeg E
 =
 \deg E+
 \sum_{p\in D}\sum_{j=1}^3
 \alpha_{p,j}\dim(F_p^j/F_p^{j+1}),
 \qquad
 \mu_{\rm par}(E)=\frac{\pardeg E}{\rank E}.
\end{equation}
A parabolic Higgs bundle is stable if every nonzero proper saturated
$\Phi$-invariant subbundle $F\subset E$, equipped with the induced
parabolic structure, satisfies
$$
 \mu_{\rm par}(F)<\mu_{\rm par}(E).
$$
Semistability is defined by the weak inequality.
\end{defn}

The sheaf $\ParEnd(E)$ consists of holomorphic endomorphisms whose
value at each $p\in D$ preserves every step of the parabolic flag.
The subsheaf $\SParEnd(E)$ consists of endomorphisms that strictly
lower every flag step.  Their trace-free subsheaves are
$$
 \begin{aligned}
 \ParEnd_0(E)&=\ker\!\left(\tr:\ParEnd(E)\to\mathcal O_X\right),\\
 \SParEnd_0(E)&=\ker\!\left(\tr:\SParEnd(E)\to\mathcal O_X\right).
 \end{aligned}
$$

\subsection{Harmonic metrics and the Hitchin metric}
\begin{defn}
A Hermitian metric $h$ on $E|_{X\setminus D}$ is adapted to the
parabolic structure when its growth filtration agrees with the given
weighted filtration.  In a local coordinate $z$ centered at $p$ and
a flag-compatible holomorphic frame $e_1,e_2,e_3$, the diagonal metric
$$
 h_{\boldsymbol\alpha}
 =
 \diag\bigl(
 |z|^{2\alpha_{p,1}},
 |z|^{2\alpha_{p,2}},
 |z|^{2\alpha_{p,3}}
 \bigr)
$$
is the standard adapted local model.
\end{defn}

Fix a determinant trivialization
$\tau:\det E\xrightarrow{\sim}\mathcal O_X$.  For a Hermitian
metric $h$, write $A^{\dagger_h}$ for the fiberwise metric adjoint of $A$.
We consider the degree-zero data with $\sum_j\alpha_{p,j}=0$ at
every puncture and normalize the determinant metric by
$|\tau^{-1}(1)|_{\det h}=1$.  The adapted harmonic metric satisfies
\begin{equation}\label{eq:hitchin-equations}
 F_{D_h}+[\Phi,\Phi^{\dagger_h}]=0,
 \qquad
 \bar{\partial}_{D_h}\Phi=0.
\end{equation}
Existence and uniqueness in the stable tame setting are supplied by
\cite{Simpson1990}; the resulting parabolic moduli metric is the
hyperk\"ahler quotient metric of \cite{Konno1993}.

Throughout, $\tr$ denotes the usual matrix trace.  We use
the induced Hermitian pairing
$$
 \langle A,B\rangle_h=\tr(AB^{\dagger_h}).
$$
Let $(a,\varphi)$ and $(b,\psi)$ be the $h$-harmonic
representatives of tangent classes in the trace-free strongly
parabolic deformation complex.  They are endomorphism-valued
$(0,1)$- and $(1,0)$-forms on $X\setminus D$ in Coulomb gauge.
The Hitchin $L^2$ metric is defined as
\begin{equation}\label{eq:hitchin-metric}
 g_{L^2}\bigl((a,\varphi),(b,\psi)\bigr)
 =
 2\operatorname{Re}\int_{X\setminus D}
 \left(
 \langle a,b\rangle_h
 +
 \langle\varphi,\psi\rangle_h
 \right)\dd A.
\end{equation}
The form indices are contracted pointwise before integration.
These representatives are supplied by the parabolic gauge theory of
\cite[Section~3.3]{Konno1993}, with the determinant and weights held fixed.

\subsection{Spectral curves and the semi-flat metric}
\begin{defn}
For a logarithmic Higgs bundle, the spectral curve is the zero scheme
$$
 \Sigma_\Phi
 =
 \left\{
 (x,\eta)\in\operatorname{Tot}(K_X(D)):
 \det(\eta-\Phi_x)=0
 \right\}.
$$
The regular locus of the Hitchin fibration is the open set on which this
curve is smooth.
\end{defn}

Over the smooth spectral locus, the BNR correspondence identifies the
underlying logarithmic Higgs bundle with a line bundle of prescribed
degree on its spectral curve.  The corresponding fiber is a Picard
torsor; fixing the determinant cuts out a fiber of the norm map
\cite[Proposition~3.6]{BNR89}.  In the present family, total ramification
recovers the full flags from the spectral line bundle, as in
Lemma~\ref{lem:spectral-local-form}.  Since $X=\Pp^1$ and
$\Pic^0(X)=0$, the regular fiber is identified with $\Pic^3(Y)$.
The line bundle $L_0$ defined in Section~\ref{subsec:regular-locus}
gives the identification
$$
 \Pic^0(Y)\xrightarrow{\sim}\Pic^3(Y),\qquad
 q\longmapsto L_0\otimes P_q.
$$

\begin{defn}
A vertical tangent vector is tangent to a smooth Hitchin fiber.  In the
parametrization of Proposition~\ref{prop:spectral}, horizontal vectors
are obtained by holding $q$ fixed; this is the Gauss--Manin splitting
with origin $L_0$.  The semi-flat metric is the standard
hyperk\"ahler metric of this algebraically completely integrable system: its
horizontal and vertical restrictions are the spectral $L^2$ pairings,
and the two distributions are orthogonal
\cite[Section~7]{FMSW22},
\cite[Proposition~2.14 and Theorem~2.15]{Fredrickson-1810.01554},
and \cite[Proposition~4.6]{Mochizuki23}.
\end{defn}

Explicit periods and flat
coefficients will be computed in Section~\ref{sec:alg}.

\subsection{The moduli space}
\label{subsec:moduli}

Let
$$
 X=\Pp^1_\C,\qquad D=\{0,1,\infty\},\qquad 0<\delta<1/3.
$$
The original weights $\boldsymbol\alpha^\delta$ are
\begin{equation}\label{eq:original-weights}
 (0,1/3,2/3+\delta)\quad\text{at }0,\qquad
 (0,1/3,2/3)\quad\text{at }1,\infty.
\end{equation}
We take $\deg E=-3$; these weights give $\pardeg E=\delta$.
Let $L_\delta$ be the filtered line bundle whose underlying holomorphic
line bundle is $\mathcal O_X(1)$, whose scalar weights are
$$
 \ell_0=-\frac{1+\delta}{3},\qquad
 \ell_1=\ell_\infty=-\frac13.
$$
Fix an isomorphism of underlying holomorphic line bundles
$$
 L_\delta^{\otimes3}\xrightarrow{\sim}\mathcal O_X(3).
$$
Its parabolic degree is
$$
 \pardeg L_\delta
 =1-\frac13-\frac13-\frac{1+\delta}{3}
 =-\frac{\delta}{3}.
$$
Adding the scalar weights gives the degree-zero normalized weights
\begin{equation}\label{eq:normalized-weights}
 \widehat{\boldsymbol\alpha}_p^\delta=
 \begin{cases}
 \left(-\dfrac13-\dfrac\delta3,-\dfrac\delta3,
       \dfrac13+\dfrac{2\delta}{3}\right),&p=0,\\[2mm]
 \left(-\dfrac13,0,\dfrac13\right),&p=1,\infty.
 \end{cases}
\end{equation}
Consequently
$$
 \deg(E\otimes L_\delta)=0,\qquad
 \pardeg(E\otimes L_\delta)=0.
$$

\begin{defn}
\label{defn:moduli-space}
The space $\calM_\delta$ is the coarse moduli space of stable,
strongly parabolic rank-three Higgs bundles
$$
 (E,\Phi,F^\bullet,\tau),\qquad
 \Phi\in H^0\!\left(
 X,\SParEnd_0(E)\otimes K_X(D)
 \right),
$$
with the degree-zero normalized weights
$\widehat{\boldsymbol\alpha}^\delta$ displayed above, $\deg E=0$,
and a chosen determinant trivialization
$\tau:\det E\xrightarrow{\sim}\mathcal O_X$.  Isomorphisms preserve
the flags and $\tau$.
\end{defn}

The normalization is compatible with families and with the deformation
pairings.  For a parameter scheme $S$, write $p_X,p_S$ for the projections
from $X\times S$ and set
$$
 N_{\mathcal E}
 =p_{S*}\bigl(\det\mathcal E\otimes p_X^*\mathcal O_X(3)\bigr)
$$
for a family with the original weights and degree $-3$.

\begin{prop}
\label{prop:chamber}
Tensoring by $L_\delta$ identifies families with the original weights and
an assigned trivialization of $N_{\mathcal E}$ with the
determinant-trivialized families defining $\calM_\delta$.  This
identification commutes with base change and preserves stability,
the strongly parabolic condition, the Hitchin map, and the trace-free
deformation complexes and their trace and $L^2$ pairings.  The interval
$0<\delta<1/3$ contains no stability wall.
\end{prop}

\begin{proof}
Every fiber of
$\det\mathcal E\otimes p_X^*\mathcal O_X(3)$ is the trivial line on
$\Pp^1$.  Cohomology and base change therefore make $N_{\mathcal E}$
a line bundle and give an evaluation isomorphism compatible with base change:
$$
 p_S^*N_{\mathcal E}
 \simeq\det\mathcal E\otimes p_X^*\mathcal O_X(3).
$$
The chosen isomorphism $L_\delta^{\otimes3}\simeq\mathcal O_X(3)$ and
a trivialization of $N_{\mathcal E}$ now trivialize
$\det(\mathcal E\otimes p_X^*L_\delta)$.  Tensoring by
$L_\delta^{-1}$ reverses this construction.  The degree and weight
calculations above give the normalized parabolic data.

Tensoring preserves the residue matrices and their flag-lowering
property.  For every invariant saturated subbundle $F$,
$$
 \mu_{\rm par}(F\otimes L_\delta)
 -\mu_{\rm par}(E\otimes L_\delta)
 =\mu_{\rm par}(F)-\mu_{\rm par}(E).
$$
It therefore preserves stability and the characteristic polynomial.
The identifications $A\mapsto A\otimes1$ of $\ParEnd_0$ and
$\SParEnd_0$ commute with the deformation differentials.  For any line
metric $h_L$,
$$
 \tr\bigl((A\otimes1)
 (B\otimes1)^{\dagger_{h\otimes h_L}}\bigr)
 =\tr(AB^{\dagger_h}),
$$
and the scalar Chern connection acts trivially on endomorphisms.  This
also identifies the $L^2$ pairings at each fixed weight.

Finally, a rank-one invariant subbundle has parabolic degree of one of
the forms
$$
 \frac{k-\delta}{3},\qquad\frac{k+2\delta}{3},\qquad k\in\Z,
$$
whereas in rank two the possibilities are
$$
 \frac{k-2\delta}{3},\qquad\frac{k+\delta}{3},\qquad k\in\Z.
$$
These follow by selecting the induced weights at the three marked
points and absorbing the integral degree into $k$.  None vanishes for
$0<\delta<1/3$, which proves the chamber assertion.
\end{proof}

In particular, every semistable object with these weights is stable.
For fixed data $(E,\Phi,F^\bullet,\tau)$, stability is independent of
$\delta\in(0,1/3)$.  The identity on these data identifies the coarse
moduli spaces, and the parabolic deformation complex is unchanged.
The adapted harmonic metrics depend on $\delta$.  These identifications
will be used throughout the analytic estimates.

Every automorphism of a stable parabolic Higgs bundle is scalar.
Hence an object of Definition~\ref{defn:moduli-space} has automorphism
group $\mu_3$.  These central automorphisms are compatible with base
change.  Rigidification by $\mu_3$ gives the algebraic space
$\calM_\delta$, and the projection from the stable moduli stack is a
$\mu_3$-gerbe \cite[Theorem~5.1.5]{ACV03}.

For a family over $S$, two determinant trivializations $\tau$ and
$f\tau$, with $f\in\Gamma(S,\mathcal O_S^\times)$, define isomorphic
objects after the \'etale base change given by $c^3=f^{-1}$:
the isomorphism is $c\Id$.  The line $N_{\mathcal E}$ is locally
trivial on $S$, so Proposition~\ref{prop:chamber} induces an isomorphism
from the coarse moduli space with the original weights and degree $-3$
to $\calM_\delta$.

We take the expansion
$$
 \det(\eta-\Phi)
 =
 \eta^3+c_1(\Phi)\eta^2+c_2(\Phi)\eta+c_3(\Phi).
$$
Nilpotence of every residue improves the possible pole orders to
$$
 c_1(\Phi)\in H^0(X,K_X),\quad
 c_2(\Phi)\in H^0(X,K_X^2(D)),\quad
 c_3(\Phi)\in H^0(X,K_X^3(2D)).
$$
The respective degrees are $-2,-1,0$.  The first two spaces vanish,
while
$$
 H^0(X,K_X^3(2D))\simeq\C.
$$
We fix its generator
$$
 q_1=\frac{dz^3}{z^2(z-1)^2}
$$
and characterize the Hitchin map by
\begin{equation}\label{eq:hitchin-map}
 h:\calM_\delta\longrightarrow\C,
 \qquad h(E,\Phi)=t,
 \qquad
 \det(\eta-\Phi)=\eta^3+tq_1.
\end{equation}
Homogeneity gives
$$
 \Phi\longmapsto a\Phi
 \quad\Longrightarrow\quad
 t\longmapsto a^3t.
$$

\subsection{The regular locus of the Hitchin fibration}
\label{subsec:regular-locus}

The spectral curve over $t=1$ is
\begin{equation}\label{eq:spectral-cover}
 Y=\overline{\{(z,\zeta):\zeta^3+z(z-1)=0\}},
 \qquad
 \pi:Y\longrightarrow X.
\end{equation}
Let
$$
 \sigma(z,\zeta)=(z,\omega\zeta)
$$
be its deck generator.
The cover is totally ramified over $D$
\cite[Proposition~2.6 and Sections~2.4--2.5]{ES26}.  Let
$r_0,r_1,r_\infty$ be its ramification points and put
$$
 R_Y=r_0+r_1+r_\infty.
$$
The valuation of $z(z-1)$ at zero is one, so the cubic polynomial
is irreducible over $\C(z)$ and $Y$ is connected.  The Riemann--Hurwitz formula gives
$$
 2g_Y-2=3(2g_X-2)+\sum_{p\in D}(3-1)=-6+6=0.
$$
The ramification divisor is $2R_Y$.  Connectedness of $Y$ will give
the scalar Poincar\'e inequality used in the coercive estimate.
We take $r_0$ as the origin of $Y$.  Let $\mathscr P$ be the
Poincar\'e line bundle on $Y\times\Pic^0(Y)$, equipped with
trivializations along $\{r_0\}\times\Pic^0(Y)$ and $Y\times\{0\}$
that agree at $(r_0,0)$
\cite[Chapter~II, Section~5]{BL92}
\cite[Definition~2.3 and Proposition~2.4]{BGHdJ16}.
For $q\in\Pic^0(Y)$, define
\begin{equation}\label{eq:spectral-lines}
 P_q=\mathscr P|_{Y\times\{q\}},\qquad
 L_0=\mathcal O_Y(R_Y),\qquad
 L_q=L_0\otimes P_q.
\end{equation}

On $Y$, put
\begin{equation}\label{eq:spectral-differential}
 \alpha
 =
 \zeta\frac{dz}{z(z-1)}
 =
 -\frac{dz}{\zeta^2}.
\end{equation}
Then
$$
 \alpha^3=-\pi^*q_1.
$$
This differential extends holomorphically across the ramification
points.  At $0$ and $1$, a parameter with $z-p$ of order three makes
both $dz$ and $\zeta^2$ vanish to order two.  At infinity, $z$ has a
pole of order three and $\zeta$ a pole of order two, so the orders in
$dz/\zeta^2$ again cancel.  Thus $\alpha$ is a nonzero holomorphic
one-form on the elliptic curve and has no zeros.  Moreover,
$\sigma^*\alpha=\omega\alpha$.
For each $p\in D$, choose a local coordinate $w_p$ on $Y$ with
$w_p(r_p)=0$ and $3\,dw_p=\alpha$, and use $x_p=w_p^3$ as the
induced base coordinate.  Thus
$$
 q_1=-x_p\left(\frac{dx_p}{x_p}\right)^3.
$$
If $a_p$ is a local holomorphic generator of $P_q$ and $s_p=w_p^{-1}a_p$,
the ordered basis
$$
 (s_p,w_ps_p,w_p^2s_p)
$$
has orders $-1,0,1$ relative to $a_p$.

The Poincar\'e bundle defines
$$
 \begin{aligned}
 \mathscr L&=\operatorname{pr}_1^*L_0\otimes\mathscr P
 &&\text{on }Y\times\Pic^0(Y),\\
 \mathscr E&=(\pi\times\Id)_*\mathscr L
 &&\text{on }X\times\Pic^0(Y).
 \end{aligned}
$$
We pull these bundles back to the product with $\C^\times$.
Multiplication by $\lambda\alpha$ equips $\mathscr E$ with the
Higgs field of the spectral family.

Since $\lambda\alpha=\lambda w_p\,dx_p/x_p$, spectral
multiplication has the following local form.

\begin{lem}
\label{lem:spectral-local-form}
In the basis $(s_p,w_ps_p,w_p^2s_p)$,
$$
 \Phi_{\lambda,q}
 =\lambda
 \begin{pmatrix}0&0&x_p\\1&0&0\\0&1&0\end{pmatrix}
 \frac{dx_p}{x_p}.
$$
Its residue is regular nilpotent, with image flag
$$
 E_p\supset\langle w_ps_p,w_p^2s_p\rangle
 \supset\langle w_p^2s_p\rangle\supset0.
$$
\end{lem}

The matrix follows immediately from $w_p^3=x_p$.  It fixes the local
normalization used in the Toda equation below.

The relative spectral correspondence now provides a parametrization of
the regular locus, including its determinant normalization.

\begin{prop}
\label{prop:spectral}
For $t=\lambda^3\neq0$, define
$$
 f_\lambda:Y\xrightarrow{\sim}\Sigma_{\lambda^3},\qquad
 f_\lambda(z,\zeta)
 =
 \left(
 z,\lambda\zeta\frac{dz}{z(z-1)}
 \right).
$$
The relative BNR correspondence and the line bundle $\mathscr P$
with its specified trivializations define the covering map
$$
 \begin{aligned}
 \widetilde\Psi_\delta:
 \C^\times\times\Pic^0(Y)&\longrightarrow\calM_\delta^{\rm reg},\\
 (\lambda,q)&\longmapsto[E_q,\Phi_{\lambda,q}].
 \end{aligned}
$$
Together with the coordinate $\lambda$, this gives a biholomorphism
\begin{equation}\label{eq:spectral-base-change}
 (\widetilde\Psi_\delta,\lambda):
 \C_\lambda^\times\times\Pic^0(Y)
 \xrightarrow{\sim}
 \calM_\delta^{\rm reg}\times_{\C_t^\times}\C_\lambda^\times.
\end{equation}
At $(\lambda,q)$, the spectral line bundle is
$(f_\lambda^{-1})^*L_q$.  The direct image and Higgs field are
$$
 E_q=\pi_*L_q,\qquad
 \Phi_{\lambda,q}
 =
 \text{multiplication by }\lambda\alpha.
$$
The family carries a determinant trivialization compatible with base change, and
$$
 \det(\eta-\Phi_{\lambda,q})
 =
 \eta^3+\lambda^3q_1.
$$
\end{prop}

\begin{proof}
Relative BNR \cite[Proposition~3.6]{BNR89},
\cite[Proposition~2.2]{Gro12} identifies spectral line bundles with
Higgs bundles.  Finite flat pushforward commutes with base change and
sends line bundles to rank-three bundles: locally, an invertible module
over a finite locally free algebra is finite projective over the base.
Thus $\mathscr E=(\pi\times\Id)_*\mathscr L$ is a family of vector bundles.
The identity $\alpha^3=-\pi^*q_1$ gives its characteristic polynomial,
and equality of Euler characteristics gives
$$
 \deg E_q
 =\deg L_q+1-g_Y-3(1-g_X)=0.
$$
The polynomial $\xi^3+\lambda^3z(z-1)$ is irreducible over $\C(z)$,
since the valuation of $z(z-1)$ at zero is one.  Hence the resulting
Higgs field has no proper invariant subbundle and is stable.
Lemma~\ref{lem:spectral-local-form} supplies its full flags and strongly
parabolic residues; assign the weights
$\widehat{\boldsymbol\alpha}^\delta$ to their successive quotients.

We verify the determinant for the whole family.  Riemann--Hurwitz gives
$\pi^*K_X(D)=K_Y(R_Y)$; using $\alpha$ to trivialize $K_Y$ identifies
this line with $L_0$.  The spectral algebra and norm formula give
$$
 \pi_*\mathcal O_Y
 =\mathcal O_X\oplus K_X(D)^{-1}\oplus K_X(D)^{-2},
 \qquad
 \det(\pi_*\mathcal O_Y)\simeq\mathcal O_X(-D),
$$
\begin{equation}\label{eq:norm-determinant}
 \Nm_\pi(L_0)=\mathcal O_X(D),\qquad
 \det(\pi_*Q)=\Nm_\pi(Q)\otimes\det(\pi_*\mathcal O_Y);
\end{equation}
see \cite[Proposition~1.7]{Car19}.  Put
$\mathscr N=\Nm_{\pi\times\Id}(\mathscr P)$ on
$X\times\Pic^0(Y)$.  Each $\mathscr N_q$ has degree zero on $\Pp^1$,
so fiberwise evaluation and base change give
$$
 \mathscr N\simeq\operatorname{pr}_2^*
 \bigl(\mathscr N|_{\{0\}\times\Pic^0(Y)}\bigr),
$$
where $\operatorname{pr}_2$ is the second projection.
The fiber of $\pi$ over zero is $3r_0$.  On a sufficiently small
parameter neighborhood, two trivializations of $\mathscr P$ over
$3r_0$ extending its specified trivialization at $r_0$ differ by a
unit $h(w)$ modulo $w^3$, with coefficients in the parameter ring
and $h(0)=1$.  Multiplication by this unit has
determinant $h(0)^3=1$.  Explicitly, for
$h(w)=1+a_1w+a_2w^2$, its matrix in the basis $(1,w,w^2)$ is
$$
 \begin{pmatrix}
 1&0&0\\ a_1&1&0\\ a_2&a_1&1
 \end{pmatrix}.
$$
Thus the norm of $\mathscr P|_{3r_0\times\Pic^0(Y)}$ has a
trivialization independent of the local extension and compatible
with base change.  It trivializes
$\mathscr N|_{\{0\}\times\Pic^0(Y)}$, and hence $\mathscr N$.  Combining the norm formula
with $\Nm(L_0)=\mathcal O_X(D)$ gives the required trivialization of
$\det\mathscr E$.

Conversely, a regular Higgs bundle has a degree-three spectral line bundle,
uniquely of the form $L_0\otimes P_q$.  The universal property of the
Poincar\'e line bundle with the specified trivialization at $r_0$
\cite[Proposition~2.2]{KU21}, together with relative BNR, makes this
correspondence holomorphic in families.  Determinant choices
are related by scalar isomorphisms locally on the parameter space, so
the correspondence induces the stated biholomorphism of coarse spaces.
\end{proof}

The Abel--Jacobi isomorphism is $s\mapsto\mathcal O_Y(s-r_0)$.
Since $\sigma(r_0)=r_0$ and $\sigma^*\alpha=\omega\alpha$,
the pullback $\sigma^*$ on $\Pic^0(Y)$ is multiplication by
$\omega^{-1}$ in this identification.  The cubic cover carries the
following descent action.

\begin{prop}
\label{prop:spectral-quotient}
The order-three action sends
\begin{equation}\label{eq:spectral-action}
 (\lambda,q)\longmapsto(\omega\lambda,\omega^{-1}q).
\end{equation}
There are determinant-preserving isomorphisms of families
$$
 \Xi_{\lambda,q}:
 (E_{\omega^{-1}q},
  \Phi_{\omega\lambda,\omega^{-1}q})
 \xrightarrow{\sim}
 (E_q,\Phi_{\lambda,q})
$$
satisfying the cocycle condition for the order-three action.  Two
parameters with $\lambda'=\omega^k\lambda$ determine the same
point of the coarse moduli space
exactly when $q'=\omega^{-k}q$.  Hence
\begin{equation}\label{eq:spectral-quotient}
 \Psi_\delta:
 (\C^\times\times\Pic^0(Y))/\mu_3
 \xrightarrow{\sim}\calM_\delta^{\rm reg}.
\end{equation}
The corresponding $\mathrm{SL}(3,\C)$-moduli stack is a
$\mu_3$-gerbe over $\calM_\delta^{\rm reg}$.
\end{prop}

\begin{proof}
The inverse deck transformation satisfies
$$
 f_{\omega\lambda}\circ\sigma^{-1}=f_\lambda.
$$
The automorphism $\sigma^{-1}$ fixes $r_0$ and preserves $R_Y$.
Its pullback on $\Pic^0(Y)$ is multiplication by $\omega$.  The
universal property of $\mathscr P$ gives a unique isomorphism
$$
 (\sigma^{-1})^*P_{\omega^{-1}q}\xrightarrow{\sim}P_q
$$
preserving the specified trivializations at $r_0$.  Tensoring with
the natural isomorphism
$(\sigma^{-1})^*\mathcal O_Y(R_Y)\simeq\mathcal O_Y(R_Y)$ gives
$$
 \iota_q:(\sigma^{-1})^*L_{\omega^{-1}q}\xrightarrow{\sim}L_q.
$$
On the degree-zero factors, the triple composite preserves the
trivialization at $r_0$ and hence equals the identity.  The natural
pullback isomorphisms of $\mathcal O_Y(R_Y)$ satisfy the same cocycle
condition, so the triple composite of the $\iota_q$ is the identity.  Pullback and finite direct image give $\Xi_{\lambda,q}$.
These isomorphisms intertwine the Higgs fields, since
$$
 (\sigma^{-1})^*\!\left(
 (\omega\lambda)\frac{\zeta\,dz}{z(z-1)}
 \right)
 =\lambda\frac{\zeta\,dz}{z(z-1)}.
$$
The induced trivialization of $\Nm_\pi(P_q)$ is preserved, and the
action on the basis $(w_p^{-1},1,w_p)$ of $L_0$ has eigenvalues
$(\omega,1,\omega^{-1})$, of product one.  Thus $\Xi_{\lambda,q}$
preserves the determinant trivialization.  It also preserves the
maximal-ideal filtration, hence the flags.  This proves the cocycle and
quotient assertions.
\end{proof}

The local form in Lemma~\ref{lem:spectral-local-form} leads directly to
the cyclic Toda model at the three totally ramified points.

\section{Cyclic Toda models}
\label{sec:local-analysis}

We construct the local harmonic metrics and establish the estimates
needed for their uniform gluing over compact weight intervals.

\subsection{The local harmonic metric}

Fix a compact interval $I\Subset(0,1/3)$.  In this subsection we write
$z=x_0$ and $w=w_0$ for the coordinates of
Lemma~\ref{lem:spectral-local-form}, so $z=w^3$.  We index the local frame
by $e_0,e_1,e_2$, with $e_j$
corresponding to the weight $\widehat\alpha_{j+1}^\delta$.
The normalized weights and the exponents in the cyclic frame are
$$
 \widehat\alpha^\delta
 =\left(-\frac13-\frac\delta3,-\frac\delta3,
        \frac13+\frac{2\delta}{3}\right),
 \qquad
 \gamma^\delta=(-\delta,-\delta,2\delta).
$$
For the local Higgs field of Lemma~\ref{lem:spectral-local-form}, the
cyclic frame $f_j=w^{1-j}\pi^*e_j$, $j=0,1,2$, gives
$$
 \Phi_\lambda
 =\lambda\begin{pmatrix}0&0&z\\1&0&0\\0&1&0\end{pmatrix}\frac{dz}{z},
 \qquad
 \pi^*\Phi_\lambda=3\lambda P\,dw,
 \qquad
 P=\begin{pmatrix}0&0&1\\1&0&0\\0&1&0\end{pmatrix}.
$$
We use $F_H=\bar\partial(H^{-1}\partial H)$ and
$\Delta_w=4\partial_w\partial_{\bar w}$.  Write
\begin{equation}\label{eq:toda-metric}
 \begin{gathered}
 H_\delta^{\rm mod}(\rho)
 =\diag(e^{2u_0^\delta(\rho)},e^{2u_1^\delta(\rho)},
        e^{2u_2^\delta(\rho)}),\\
 \sum_i u_i^\delta=0,\qquad \rho=3|\lambda||w|.
 \end{gathered}
\end{equation}
The curvature is $F_{H,ii}=-\frac12\Delta_wu_i\,dw\wedge d\bar w$, while
$[\Phi,\Phi^{\dagger_H}]_{ii}$ has coefficient
$9|\lambda|^2(e^{2(u_i-u_{i-1})}-e^{2(u_{i+1}-u_i)})$.
The Hitchin equation therefore becomes
\begin{equation}\label{eq:cyclic-toda}
 (\partial_\rho^2+\rho^{-1}\partial_\rho)u_i
 =2\bigl(e^{2(u_i-u_{i-1})}-e^{2(u_{i+1}-u_i)}\bigr),
 \qquad i\in\Z/3\Z.
\end{equation}

The grading action is $e_j\mapsto\omega^j e_j$.

\begin{thm}
\label{thm:toda}
There is a unique grading-invariant determinant-one harmonic metric
$h_{\lambda,\delta}^{\rm mod}$ on $\C^\times$, adapted to
$\widehat\alpha^\delta$ at zero and asymptotic to the decoupled
metric at infinity.  Its pullback
to the cubic cover is invariant under the deck group and the lifted
circle action.  In the cyclic frame, its diagonal functions solve
\eqref{eq:cyclic-toda} with
\begin{equation}\label{eq:toda-boundary-conditions}
 \begin{aligned}
 u_i^\delta(\rho)&=\gamma_i^\delta\log\rho+O(1)
                    &&(\rho\downarrow0),\\
 u_i^\delta(\rho)&\longrightarrow0 &&(\rho\to\infty).
 \end{aligned}
\end{equation}
The diagonal representative is independent of $\arg\lambda$.
\end{thm}

\begin{proof}
For $\lambda=1$, the existence and uniqueness theorem of
Li--Mochizuki for grading-invariant metrics applies to the cubic
differential $z^{-2}dz^3$; see
\cite[Theorem~1.10, Corollary~1.11 and Propositions~3.19,
3.21--3.22]{LM2010b}.  Under the identification
$e_j=(dz/z)^{1-j}$, their frame
$\kappa_i=(dz)^{2-i}$ satisfies
$\kappa_i=z^{2-i}e_{i-1}$, $i=1,2,3$, so their growth convention gives
$b_i^\delta=-2-\widehat\alpha_i^\delta$.
Their integers at zero and infinity are $1$ and $-1$, respectively,
and their parameter at zero is
$$
 b^\delta
 =\left(\frac{-5+\delta}{3},\frac{-6+\delta}{3},
              \frac{-7-2\delta}{3}\right),\qquad
 \sum_i b_i^\delta=-6.
$$
The three inequalities defining the relevant polytope are strict:
$$
 b_1-b_2=\frac13,\qquad b_2-b_3=\frac13+\delta,
 \qquad b_3-(b_1-1)=\frac13-\delta.
$$
Thus each block in the cited asymptotic decomposition has one index,
and the expansion as $\rho\to0$ has no $\log(-\log\rho)$ terms.
In the cyclic frame the exponents are
$3\widehat\alpha_{j+1}^\delta+1-j=\gamma_j^\delta$, $j=0,1,2$.
The grading symmetry makes the summands orthogonal.  For
$\nu\in S^1$, set $R_\nu(z)=\nu^3z$ and
$D_\nu=\diag(\nu,1,\nu^{-1})$.  Then
$$
 D_\nu^{-1}R_\nu^*\Phi_1D_\nu=\nu\Phi_1.
$$
This lifted circle action preserves Hitchin's equation, the grading,
and both endpoint conditions.  Uniqueness gives radial symmetry.
Rescaling by $\rho=3|\lambda||w|$ gives the stated equation and
normalization for every $\lambda\ne0$.  Multiplication of the Higgs
field by a unit complex number preserves the metric by the same
uniqueness argument.
\end{proof}

\subsection{Uniform endpoint estimates}

Set $b_I=1-3\sup I>0$.
We obtain uniform bounds for the constant terms at zero and the
exponentially decaying tails by applying the implicit function theorem
in fixed weighted spaces.  These estimates also control the boundary
terms in the metric comparison.

\begin{defn}\label{defn:b-holder-spaces}
Let $w=re^{\ii\vartheta}$ be a coordinate on a punctured disk
$0<r<\varrho$, and let $k\in\Z_{\geq0}$ and $0<\alpha<1$.  The b-H\"older
space $C_b^{k,\alpha}$ consists of functions whose expressions in
$(s,\vartheta)=(\log r,\vartheta)$ have finite uniform
$C^{k,\alpha}$ norm for the product metric $ds^2+d\vartheta^2$.
Thus derivatives are taken using $r\partial_r$ and
$\partial_\vartheta$, and the H\"older seminorms use cylindrical
distance.  For $\nu\in\R$, set
\begin{equation}\label{eq:b-holder-norm}
 r^\nu C_b^{k,\alpha}
 =\{f:r^{-\nu}f\in C_b^{k,\alpha}\},\qquad
 \|f\|_{r^\nu C_b^{k,\alpha}}
 =\|r^{-\nu}f\|_{C_b^{k,\alpha}}.
\end{equation}
The subscript $\mathrm{rad}$ denotes the closed subspace of radial
functions.  These definitions apply componentwise to vector-valued
functions and, in specified frames, to sections of vector bundles.
We write $f=O_b(r^\nu)$ when every
$(r\partial_r)^j\partial_\vartheta^\ell(r^{-\nu}f)$ is bounded.
\end{defn}

The vector fields in this definition are tangent to the boundary
circle of the real blowup at the puncture.  This is the b-H\"older
regularity used in \cite[Section~5.1]{FMSW22}.  For $\tau>0$ on a
radial end $\rho\geq1$, the notation
$e^{-\tau\rho}C_{\mathrm{eucl,rad}}^{k,\alpha}$ means
\begin{equation}\label{eq:radial-exponential-norm}
 \|f\|_{e^{-\tau\rho}C_{\mathrm{eucl,rad}}^{k,\alpha}}
 :=\sup_{a\geq1}
 \|e^{\tau\rho}f\|_{C^{k,\alpha}([a,a+1])}<\infty.
\end{equation}
Here the derivatives and H\"older seminorms use the Euclidean variable
$\rho$.  For radial functions this norm is equivalent to the uniform
Euclidean H\"older norm on annuli of fixed width.

\begin{prop}
\label{prop:toda-decay}
The constants $c_i(\delta)$ in the expansion as $\rho\to0$ are continuous in
$\delta$ and sum to zero.  For every integer $a\geq0$,
\begin{equation}\label{eq:toda-origin-estimate}
 \left| (\rho\partial_\rho)^a
 \bigl(u_i^\delta-\gamma_i^\delta\log\rho-c_i(\delta)\bigr)\right|
 \le C_{I,a}\rho^{b_I},
 \qquad 0<\rho\le1.
\end{equation}
For every $0<\sigma<2\sqrt3$ and every integer $a\geq0$,
\begin{equation}\label{eq:toda-infinity-estimate}
 |(\rho\partial_\rho)^a u_i^\delta(\rho)|
 \le C_{I,\sigma,a}e^{-\sigma\rho},\qquad \rho\ge1.
\end{equation}
These estimates hold uniformly for $\delta\in I$.  For a fixed smooth
cutoff $\chi$ equal to one for $\rho\leq1$ and zero for $\rho\geq2$,
the functions $u^\delta-\chi\gamma^\delta\log\rho$ depend continuously
on $\delta$ in the norm \eqref{eq:toda-solution-norm}, locally in
$\delta$ with $\eta$ and $\tau$ fixed as in the proof.
Substituting $\rho=3|\lambda||w|$ gives the corresponding
spatial derivative bounds in the $w$-coordinate.
\end{prop}

\begin{proof}
The Toda potential on $V=\{x\in\R^3:\sum_i x_i=0\}$ and the
linearized radial operator are
$$
 \calV(u)=\sum_i e^{2(u_i-u_{i-1})},\qquad
 \calL_\delta=-\Delta_\rho+D^2\calV(u^\delta).
$$
For fixed $\delta$, let $v_i=u_i^\delta-\gamma_i^\delta\log\rho$ near zero.
The equation gives
\begin{equation}\label{eq:toda-radial-remainder}
 (\rho v_i')'=O(\rho^{1-6\delta}).
\end{equation}
Boundedness of $v_i$ excludes a logarithmic homogeneous term, so two
integrations give
$$
 v_i=c_i(\delta)+O(\rho^{2-6\delta}).
$$
We establish weighted continuity of the solutions before making these
estimates uniform.

Fix $0<\eta<b_I$, $0<\alpha<1$, and $\delta_0\in I$.
Choose $\tau>0$ smaller than both $2\sqrt3$ and the decay rate at
$\delta_0$ in \cite[Proposition~3.19]{LM2010b}.  We use the cutoff
$\chi$ from the statement.  The radial unknowns
$v\in C_{\mathrm{loc}}^{2,\alpha}((0,\infty);V)$ have a limit
$v_0\in V$ at zero and finite norm
\begin{equation}\label{eq:toda-solution-norm}
 |v_0|
 +\|\chi(v-v_0)\|_{\rho^\eta C_{b,\mathrm{rad}}^{2,\alpha}}
 +\|(1-\chi)v\|_{e^{-\tau\rho}C_{\mathrm{eucl,rad}}^{2,\alpha}}.
\end{equation}
The radial sources $f\in C_{\mathrm{loc}}^{0,\alpha}((0,\infty);V)$
have finite norm
\begin{equation}\label{eq:toda-source-norm}
 \|\chi f\|_{\rho^{\eta-2}C_{b,\mathrm{rad}}^{0,\alpha}}
 +\|(1-\chi)f\|_{e^{-\tau\rho}C_{\mathrm{eucl,rad}}^{0,\alpha}}.
\end{equation}
The b-H\"older norms here are taken on $0<\rho<2$; the
Euclidean norms are those of \eqref{eq:radial-exponential-norm}.
These are Banach spaces.  The first allows a constant term at zero,
with a remainder in $\rho^\eta C_{b,\mathrm{rad}}^{2,\alpha}$;
its derivative has finite Dirichlet energy.  We prove that for every
source with finite norm \eqref{eq:toda-source-norm}, the equation
$\calL_{\delta_0}v=f$ has a unique solution with finite norm
\eqref{eq:toda-solution-norm}, bounded by a constant times
\eqref{eq:toda-source-norm}.
The three-cycle identity gives
\begin{equation}\label{eq:toda-linearized-mass}
 D^2\calV(0)[v,v]=4\sum_i(v_i-v_{i-1})^2=12|v|^2,
 \qquad v\in V.
\end{equation}
The constant operator $-\Delta_\rho+12$ has radial Green kernel
\begin{equation}\label{eq:toda-green-kernel}
 G(\rho,s)=I_0(2\sqrt3\min\{\rho,s\})
           K_0(2\sqrt3\max\{\rho,s\}),
\end{equation}
integrated against $s\,ds$; the Wronskian normalization and endpoint
expansions are given in \cite[Sections~10.28, 10.30 and~10.40]{DLMF}.
The constant term at zero is controlled as follows.  Set $\mu=2\sqrt3$.
For a radial source $f$ with finite norm
\eqref{eq:toda-source-norm}, the solution is
$$
 v(\rho)=K_0(\mu\rho)\int_0^\rho I_0(\mu s)f(s)s\,ds
       +I_0(\mu\rho)\int_\rho^\infty K_0(\mu s)f(s)s\,ds.
$$
The integrals converge because $\eta>0$ and $0<\tau<\mu$.
In particular,
$$
 v(0)=\int_0^\infty K_0(\mu s)f(s)s\,ds.
$$
Subtracting this constant before estimating gives
$$
 \begin{aligned}
 v(\rho)-v(0)
 &=\bigl(I_0(\mu\rho)-1\bigr)v(0)\\
 &\quad+\int_0^\rho
 \bigl(K_0(\mu\rho)I_0(\mu s)-I_0(\mu\rho)K_0(\mu s)\bigr)
 f(s)s\,ds.
 \end{aligned}
$$
The small-argument expansions of the Bessel functions bound the kernel
in the last integral by $C(1+|\log(s/\rho)|)$ for
$0<s<\rho<1$.  If $|f(s)|\leq Fs^{\eta-2}$, this yields
$$
 |v(\rho)-v(0)|
 \leq C\rho^2|v(0)|
   +CF\rho^\eta\int_0^1(1+|\log t|)t^{\eta-1}\,dt.
$$
Here the last integral is $\eta^{-1}+\eta^{-2}$, and $\eta<2$.
The radial equation then bounds $\rho v'$ and
$(\rho\partial_\rho)^2v$ by $C\rho^\eta$ times the norm
\eqref{eq:toda-source-norm}.  Schauder estimates on rescaled annuli give the weighted
H\"older estimate.  Thus the inverse has a constant term and a
$\rho^\eta C_b^{2,\alpha}$ remainder.

At infinity the large-argument expansions give
$$
 |G(\rho,s)|\leq
 C(\rho s)^{-1/2}e^{-\mu|\rho-s|},\qquad \rho,s\geq1.
$$
After multiplication by $e^{\tau\rho}$, integration against
$e^{-\tau s}s\,ds$ is uniformly bounded for $\rho\geq1$ because
$\tau<\mu$.  The contribution from $s<1$ is a constant multiple of
$K_0(\mu\rho)$.  Local Schauder estimates on annuli of fixed Euclidean width give
$e^{\tau\rho}v\in C_{\mathrm{eucl,rad}}^{2,\alpha}$ at infinity.
Together with the constant term and weighted remainder at zero, this
proves that the inverse of $-\Delta_\rho+12$ maps sources with norm
\eqref{eq:toda-source-norm} continuously to solutions with norm
\eqref{eq:toda-solution-norm}.
  At zero, the coefficient difference
$D^2\calV(u^{\delta_0})-12\Id$ is $O_b(\rho^{-6\delta_0})$; at
infinity it decays exponentially.  The strict gain
$2-\eta-6\delta_0>0$ and compact H\"older embeddings on annuli make
multiplication by this difference compact from the solution space
with norm \eqref{eq:toda-solution-norm} to the source space with norm
\eqref{eq:toda-source-norm}.  Thus $\calL_{\delta_0}$ is Fredholm
of index zero between these two specified spaces.  Moreover,
$$
 D^2\calV(u)[v,v]
 =4\sum_i e^{2(u_i-u_{i-1})}(v_i-v_{i-1})^2>0
 \quad(v\in V\setminus\{0\}).
$$
For a kernel element, integration by parts has vanishing boundary
terms: at zero $v=v_0+O_b(\rho^\eta)$, and at infinity $v$ decays
exponentially.  More explicitly, the radial boundary term
$\rho\langle v',v\rangle$ is $O(\rho^\eta)$ at zero and tends to
zero at infinity.  Integration on $(\epsilon,R)$ and passage to the
two limits therefore gives
$$
 \int_0^\infty
 \left(|v'|^2+D^2\calV(u^{\delta_0})[v,v]\right)\rho\,d\rho=0.
$$
Positivity gives $v=0$, proving the claim.

Substitution of $u=\chi\gamma^\delta\log\rho+v$ into
$-\Delta_\rho u+\nabla\calV(u)=0$ gives a continuously differentiable
map from a neighborhood of $u^{\delta_0}-\chi\gamma^{\delta_0}\log\rho$
in the space \eqref{eq:toda-solution-norm} to the space
\eqref{eq:toda-source-norm}, with $\delta$ in a neighborhood of
$\delta_0$.  Indeed, the most singular coefficient and its first
parameter derivative are bounded by
$C\rho^{-6\delta}(1+|\log\rho|)$, which the positive gap
$2-\eta-6\sup I>0$ absorbs, after slightly enlarging $I$ if necessary.
The derivative in $v$ is $\calL_{\delta_0}$.  The implicit function
theorem and uniqueness in Theorem~\ref{thm:toda} therefore give local
continuity of $u^\delta-\chi\gamma^\delta\log\rho$ in
\eqref{eq:toda-solution-norm}, including its constant term.  Uniqueness identifies the local solution maps on
overlaps.  A finite cover of $I$ therefore gives uniform bounds near
zero and an initial exponential rate $\tau_*>0$ independent of
$\delta$ at infinity.

In particular, the right-hand side of
\eqref{eq:toda-radial-remainder} is bounded by $C_I\rho^{1-6\delta}$.  Integrating twice, then using
the Toda equation recursively, gives uniform bounds for every
$(\rho\partial_\rho)^a$ of the remainder.  Since
$2-6\delta>b_I$, these imply the asserted estimates as $\rho\to0$.
At infinity the equation becomes
$$
 (-\Delta_\rho+12)u^\delta=O(|u^\delta|^2).
$$
Applying the cutoff gives
\begin{equation}\label{eq:toda-exterior-equation}
 \begin{aligned}
 (-\Delta_\rho+12)\bigl((1-\chi)u^\delta\bigr)
 &=(1-\chi)O(|u^\delta|^2)\\
 &\quad+[-\Delta_\rho,1-\chi]u^\delta.
 \end{aligned}
\end{equation}
The commutator term is supported in a fixed annulus and is uniformly
bounded for $\delta\in I$.
Variation of constants using the Green kernel
\eqref{eq:toda-green-kernel} improves a uniform rate $c$ to every rate below $\min\{2\sqrt3,2c\}$.  Finitely many
iterations reach any prescribed $\sigma<2\sqrt3$ uniformly on $I$.
Ordinary Schauder estimates give spatial derivatives; choosing a
slightly larger rate first absorbs the powers of $\rho$ in
$(\rho\partial_\rho)^a$.  Radial symmetry and rescaling give the
stated spatial derivative estimates.
\end{proof}

\section{Approximate and exact solutions}
\label{sec:global-analysis}

We construct an approximate solution from the metric on the spectral
line bundle and the cyclic Toda solutions, then correct it to a harmonic
metric by a small self-adjoint endomorphism.  Following the weighted H\"older
formulation of \cite[Sections~5.2 and~6]{FMSW22}, we work below the first
positive indicial root.  A uniform coercive estimate and local elliptic
regularity give the polynomial inverse bound needed for the correction.
Throughout this section $I\Subset(0,1/3)$ is fixed, $q\in \Pic^0(Y)$, and
$\lambda=Te^{\ii\phi}$ with $T\geq1$.

\subsection{The approximate metric}

Put $Y^\circ=Y\setminus R_Y$ and $X^\circ=X\setminus D$.
For each $q\in\Pic^0(Y)$, the degree-zero line bundle
$P_q=L_q(-R_Y)$ carries a unique Hermitian metric with flat Chern
connection for which the specified trivialization at $r_0$ is an
isometry.  These metrics form a smooth Hermitian metric on $\mathscr P$
\cite[Proposition~2.8(3) and Remark~2.9]{BGHdJ16}.
The canonical section of $\mathcal O_Y(R_Y)$
identifies $P_q$ with $L_q$ on $Y^\circ$.  Orthogonal direct image,
followed by determinant normalization, defines a metric $h_q^{\rm dec}$
on $E_q|_{X^\circ}$.

\begin{lem}\label{lem:decoupled-pair}
The metric $h_q^{\rm dec}$ has determinant one and parabolic weights
$(-1/3,0,1/3)$ at each point of $D$.  On each punctured local cover,
the cyclic frame constructed from a unit-norm flat frame of $P_q$
is orthonormal for $h_q^{\rm dec}$.  Its Chern connection satisfies
$$
 F_{A_q^{\rm dec}}=0,\qquad
 [\Phi_{\lambda,q},\Phi_{\lambda,q}^{\dagger_{h_q^{\rm dec}}}]=0.
$$
The metric varies smoothly with $q$ and is preserved by the spectral
isomorphisms of Proposition~\ref{prop:spectral-quotient}.
\end{lem}

\begin{proof}
In $x_p=w_p^3$, choose a unitary flat frame $a_p$ of $P_q$, and write
$s_p=w_p^{-1}a_p$ and $e_j=\pi_*(w_p^js_p)$, $j=0,1,2$.  Then
$$
 h_q^{\rm dec}(e_i,e_j)
 =\frac13\sum_{\ell=0}^2(\omega^\ell w_p)^{i-1}
       \overline{(\omega^\ell w_p)^{j-1}}
 =\delta_{ij}|w_p|^{2(i-1)}.
$$
The sheet sum has determinant $3^3$ in the determinant trivialization
of Proposition~\ref{prop:spectral}; the factor $1/3$ therefore gives
determinant one.  The frame $f_j=w_p^{1-j}\pi^*e_j$ makes the metric
the identity.  On the unramified cover the eigenlines are flat and
orthogonal, proving both equations.  Smooth dependence follows from the smooth Hermitian metric on
$\mathscr P$.  The isomorphisms of the degree-zero factors in
Proposition~\ref{prop:spectral-quotient} preserve the specified
trivializations at $r_0$.  They preserve the flat metrics by uniqueness,
so the induced metrics $h_q^{\rm dec}$ are invariant as well.
\end{proof}

We will also use the induced connection on the centralizer.  If $M_f$
denotes multiplication by a function $f$ on $Y^\circ$, then
$$
 \nabla^{A_q^{\rm dec}}M_f=M_{df},\qquad
 \tr M_f=\operatorname{Tr}_\pi f.
$$
Thus $f\mapsto M_f$ identifies the trace-free centralizer with
$\ker\operatorname{Tr}_\pi$, and its induced connection is independent
of $q$.

\begin{defn}\label{defn:approximate-metric}
Choose a fixed radial cutoff $\chi_0$, supported in the $w_0$-disk
and equal to one on a smaller disk.  In the cyclic frame set
\begin{equation}\label{eq:approximate-metric}
 h_{\lambda,q,\delta}^{\rm app}
 =\exp\bigl(\chi_0\log H_\delta^{\rm mod}(3T|w_0|)\bigr).
\end{equation}
On the complement the metric is $h_q^{\rm dec}$.  At $p=1,\infty$
the logarithmic exponents vanish, and $u_i=0$ gives the local solution
$H=\Id$, equal to $h_q^{\rm dec}$ in the cyclic frame.
\end{defn}

The interpolation has trace-free logarithm and respects deck
equivariance.  It consequently defines a determinant-one metric on
$E_q$, adapted to the prescribed weights.  It is an exact solution
near $D$ and outside the transition annulus.

For the linear analysis fix a conformal metric $g_c$ on $X^\circ$
such that, in $x_p=w_p^3$,
\begin{equation}\label{eq:conic-metric}
 g_c=\frac{|dx_p|^2}{9|x_p|^{4/3}},\qquad
 \pi^*g_c=|dw_p|^2.
\end{equation}
In this section, $L^2$ norms and formal adjoints use $g_c$ and the
relevant bundle metric; fiberwise adjoints use only the bundle metric.
The pairing of deformation one-forms is conformally
invariant, so this choice leaves the Hitchin metric unchanged.  A base
integral is $1/3$ of the integral of its pullback on the cubic cover.
Writing $a=3T$ and $\rho=a|w_p|$, the area form scales by $a^{-2}$,
and the lifted model operator in the unscaled coordinate $w_p$ is
$a^2$ times the normalized operator $\mathcal L_{\delta,p}^{\rm mod}$.  In this normalization
the Higgs form is $P\,d\widehat w$, with $\widehat w=3\lambda w_p$,
and the scalar principal coefficient of the operator is one.

For a trace-free $h^{\rm app}$-self-adjoint endomorphism $\gamma$, put
\begin{equation}\label{eq:metric-equation}
 \begin{aligned}
 h_\gamma&=h^{\rm app}e^{2\gamma},\\
 \mathcal F_{\lambda,q,\delta}(\gamma)
 &=\operatorname{Ad}(e^\gamma)
   \left[\ii\star_c\bigl(F_{D_{h_\gamma}}
   +[\Phi_{\lambda,q},\Phi_{\lambda,q}^{\dagger_{h_\gamma}}]\bigr)\right].
 \end{aligned}
\end{equation}
Since $\det h_\gamma=1$, the curvature has trace zero; the commutator
also has trace zero.  Conjugation by $e^\gamma$ makes
$\mathcal F_{\lambda,q,\delta}(\gamma)$ self-adjoint with respect
to $h^{\rm app}$.  We write
$$
 R_{\lambda,q,\delta}=\mathcal F_{\lambda,q,\delta}(0),\qquad
 \mathcal L_{\lambda,q,\delta}=D\mathcal F_{\lambda,q,\delta}(0).
$$

Fix $\alpha\in(0,1)$.  We use Definition~\ref{defn:b-holder-spaces}
on each punctured $w_p$-disk.  For $k\geq0$ and $\nu\in\R$, the
notation $|w|^\nu C_b^{k,\alpha}$ denotes trace-free
$h^{\rm app}$-self-adjoint sections of $\operatorname{End}E_q$ on
$X^\circ$ whose lifts belong to
$|w_p|^\nu C_b^{k,\alpha}$ near each puncture and have ordinary
$C^{k,\alpha}$ regularity on a fixed compact interior.
The coefficients are taken in the $h^{\rm app}$-unitary frames
obtained by rescaling the cyclic holomorphic frames, and satisfy the
deck-equivariance condition.  A fixed partition of unity combines
these local norms.
For $A_r=\{r<|w_p|<2r\}$, the local weighted norm has the equivalent
expression
$$
 \sup_{0<r<\varrho/2}
 r^{-\nu}\|f(r\,\cdot)\|_{C^{k,\alpha}(\{1<|\zeta|<2\})},
$$
together with an ordinary norm on the outer annulus.  Fixed
enlargements of these annuli give equivalent norms and will be used
in local elliptic estimates.  We choose
$$
 0<\widehat\eta<\frac{1-3\sup I}{2}.
$$

\begin{prop}\label{prop:residual-decay}
There are $c_I,T_I>0$ such that, for every integer $k\geq0$,
there is $C_{I,k}>0$ with
\begin{equation}\label{eq:residual-decay}
 \|R_{\lambda,q,\delta}\|_{|w|^{\widehat\eta-2}C_b^{k,\alpha}}
 +\|R_{\lambda,q,\delta}\|_{L^\infty}
 \leq C_{I,k}e^{-c_IT},\qquad T\geq T_I,
\end{equation}
uniformly on $\Pic^0(Y)\times I\times S^1$.
\end{prop}

\begin{proof}
The residual is supported in the fixed annulus
$\operatorname{supp}d\chi_0$.  Every term contains a spatial derivative
of $\log H_\delta^{\rm mod}(3T|w_0|)$ or a smooth coefficient
vanishing with this logarithm.  Proposition~\ref{prop:toda-decay}
bounds these terms by a polynomial in $T$ times
$e^{-3\sigma T|w_0|}$, uniformly on $I$, for
$0<\sigma<2\sqrt3$.  The weighted and ordinary norms are equivalent
on this fixed annulus.  If
$r_{\rm tr}=\inf_{\operatorname{supp}d\chi_0}|w_0|$, any
$0<c<3\sigma r_{\rm tr}$ absorbs the polynomial factors.
A finite cover of $\Pic^0(Y)$ by parameter neighborhoods with smooth
local identifications of $E_q$, together with the smooth metric on
$\mathscr P$, gives constants uniform in $q$.  The estimates are
independent of $\arg\lambda$.
\end{proof}

\subsection{The linearized equation}

The positivity of the linearization supplies weak solutions.  We then
obtain their weighted regularity on disks of radius comparable to
$T^{-1}$ and on the remaining coordinate annuli.

\begin{lem}\label{lem:fixed-linearization}
For the unitary pair $(A,\Phi)$ defined by $h^{\rm app}$,
\begin{equation}\label{eq:linearization-form}
 \begin{aligned}
 \mathcal L&=d_A^*d_A+2\ad_\Phi^*\ad_\Phi,\\
 \langle\mathcal Lu,u\rangle_{L^2}
 &=\|d_Au\|_{L^2}^2+2\|[\Phi,u]\|_{L^2}^2.
 \end{aligned}
\end{equation}
Here $\ad_\Phi^*$ is the real adjoint on Hermitian endomorphisms;
the second identity holds first for compactly supported sections.
\end{lem}

\begin{proof}
For $h_s=he^{2su}$, the variations are
$$
 \dot A^{1,0}=2\partial_hu,\qquad
 \dot F=2\bar\partial_E\partial_hu,\qquad
 \dot\Phi^{\dagger_h}=2[\Phi^{\dagger_h},u].
$$
The K\"ahler identities give
$d_A^*d_Au+2\ad_\Phi^*\ad_\Phi u+[R,u]$ before differentiating
the conjugation in $\mathcal F$.  That conjugation contributes
$[u,R]$, so the last terms cancel.  Integration by parts proves the
quadratic identity; compare \cite[equation~(5.2)]{FMSW22}.
\end{proof}

Let $\mathfrak q_{\lambda,q,\delta}$ be the closed quadratic form
obtained from \eqref{eq:linearization-form}.  Its form domain is
\begin{equation}\label{eq:friedrichs-form-domain}
 \operatorname{Dom}(\mathfrak q_{\lambda,q,\delta})
 =\overline{\{u\in C_c^\infty(X^\circ,\operatorname{End}_0E_q):
                    u^{\dagger_{h^{\rm app}}}=u\}}
 ^{\, (\|u\|_{L^2}^2+\mathfrak q(u,u))^{1/2}}.
\end{equation}
The norms use $g_c$ and $h^{\rm app}$.  For fixed parameters we
write $\mathfrak q=\mathfrak q_{\lambda,q,\delta}$.
The associated Friedrichs realization acts in $L^2$ on
$$
 \{u\in\operatorname{Dom}(\mathfrak q):
                  \mathcal Lu\in L^2(X^\circ,\operatorname{End}_0E_q)\},
$$
where $\mathcal Lu$ is evaluated distributionally.

\begin{defn}\label{defn:holder-domain}
Fix cutoffs $\psi_p$ equal to one near the punctures and a compact
interior $K$.  For fixed parameters, $D_{\rm Fr}$ is the Banach space
of trace-free $h^{\rm app}$-self-adjoint endomorphisms of $E_q$ whose
coefficients in the unitary frames specified above satisfy
\begin{equation}\label{eq:holder-domain}
 u=\psi_pc_p+\widetilde u_p,\qquad
\widetilde u_p\in |w_p|^{\widehat\eta}C_b^{2,\alpha},
\end{equation}
where $c_p$ is a constant trace-free diagonal Hermitian matrix.
On $K$ the sections have ordinary $C^{2,\alpha}$ regularity.  The norm is
\begin{equation}\label{eq:holder-domain-norm}
 \|u\|_{D_{\rm Fr}}
 =\|u\|_{C^{2,\alpha}(K)}
   +\sum_p|c_p|
   +\sum_p\|\widetilde u_p\|_{|w_p|^{\widehat\eta}C_b^{2,\alpha}}.
\end{equation}
\end{defn}

The space $D_{\rm Fr}$ is the rank-three analogue of the H\"older
pseudo-Friedrichs domain in \cite[Section~5.2]{FMSW22}, with
$\widehat\eta$ below the first positive indicial root.
Theorem~\ref{thm:uniform-inverse} shows that a weak solution in
$\operatorname{Dom}(\mathfrak q)$ with
$\mathcal Lu\in |w|^{\widehat\eta-2}C_b^{0,\alpha}$ belongs to
$D_{\rm Fr}$.
With the norm \eqref{eq:holder-domain-norm}, $D_{\rm Fr}$ is a Banach
space that embeds continuously into $L^\infty$, uniformly in the parameters.
Every section of $D_{\rm Fr}$ also belongs to
$\operatorname{Dom}(\mathfrak q)$.
Indeed, the covariant derivative of the remainder has square integral
bounded by a multiple of $\int_0^{\varrho}r^{2\widehat\eta-1}\,dr$.
The diagonal connection commutes with $c_p$, and the most singular
Higgs contribution has square integral bounded by a multiple of
$\int_0^{\varrho}r^{1-6\delta_p}\,dr$, where
$\delta_0=\delta$ and $\delta_1=\delta_\infty=0$.
Both integrals converge.  Logarithmic cutoffs with Dirichlet energy
tending to zero approximate the bounded section by compactly supported
ones in the norm
$(\|u\|_{L^2}^2+\mathfrak q(u,u))^{1/2}$.

The same coefficient calculation gives a bounded map

\begin{equation}\label{eq:linearized-holder-map}
 \mathcal L_{\lambda,q,\delta}:D_{\rm Fr}
 \longrightarrow |w|^{\widehat\eta-2}C_b^{0,\alpha}.
\end{equation}
In particular,
$$
 |\Phi|_{h^{\rm app}}^2
 \leq C\bigl(T^2+T^{2-6\delta_p}|w_p|^{-6\delta_p}\bigr),
 \qquad 2-6\sup I-\widehat\eta>0.
$$
Thus $\mathcal L(\psi_pc_p)\in
|w_p|^{\widehat\eta-2}C_b^{0,\alpha}$ near each puncture, and all
coefficient bounds needed here grow at most polynomially in $T$.

\begin{prop}\label{prop:uniform-coercivity}
There are $c_I>0$ and $T_0\geq1$ such that
\begin{equation}\label{eq:uniform-coercivity}
 \mathfrak q_{\lambda,q,\delta}(u,u)
 \geq c_I\|u\|_{L^2}^2,\qquad T\geq T_0,
\end{equation}
for every $u\in\operatorname{Dom}(\mathfrak q_{\lambda,q,\delta})$,
uniformly on $\Pic^0(Y)\times I\times S^1$.
\end{prop}

\begin{proof}
Suppose that there are $T_j\to\infty$, parameters
$(q_j,\delta_j,\phi_j)$, and sections
$u_j\in\operatorname{Dom}(\mathfrak q_{\lambda_j,q_j,\delta_j})$ with
$\|u_j\|_{L^2}=1$ and $\mathfrak q_j(u_j,u_j)\to0$, where
$\lambda_j=T_je^{\ii\phi_j}$ and
$\mathfrak q_j=\mathfrak q_{\lambda_j,q_j,\delta_j}$.
Remove fixed small coordinate disks $|w_p|<\varrho_p/2$ and denote
the remaining compact surface by $X_{\boldsymbol\varrho}$.
Its inverse image $Y_{\boldsymbol\varrho}$ is the connected elliptic
curve $Y$ with three disks removed.

On this fixed exterior the approximate pair is exponentially close,
in every fixed spatial norm, to
$(A_{q_j}^{\rm dec},T_j\varphi_{\phi_j,q_j}^{\rm dec})$,
where $\varphi_{\phi,q}^{\rm dec}=T^{-1}\Phi_{\lambda,q}$ is
expressed in the decoupled unitary frame.  Let $\Pi_j$ be the
orthogonal projection onto its trace-free centralizer.
The separated eigenvalues on this compact set imply
$$
 \|(1-\Pi_j)u_j\|_{L^2(X_{\boldsymbol\varrho})}
 \leq CT_j^{-1}\|[\Phi_j,u_j]\|_{L^2}+o(1)\longrightarrow0.
$$
Write $\Pi_ju_j=M_{f_j}$.  The decoupled connection preserves
$\Pi_j$, and the approximate connection differs exponentially from
it.  Consequently
$$
 \|df_j\|_{L^2(Y_{\boldsymbol\varrho})}
 \leq C\|d_{A_j}u_j\|_{L^2}+o(1)\longrightarrow0.
$$
Since $\operatorname{Tr}_\pi f_j=0$, integration with the pullback
measure gives $\int_{Y_{\boldsymbol\varrho}}f_j=0$.
The Neumann Poincar\'e inequality on this fixed connected surface
therefore gives $\|f_j\|_{L^2}\to0$.  The identification of the
centralizer with scalar functions is independent of $q_j$, so its
constant is uniform.  We have proved
$\|u_j\|_{L^2(X_{\boldsymbol\varrho})}\to0$.

For each puncture take $\chi_p$ supported in $|w_p|<\varrho_p$,
equal to one on $|w_p|<\varrho_p/2$, and zero on the outer circle.
Kato's inequality and the scalar disk Poincar\'e inequality give
$$
 \|\chi_pu_j\|_{L^2}^2
 \leq C\varrho_p^2\mathfrak q_j(\chi_pu_j,\chi_pu_j)
 \leq C\bigl(\mathfrak q_j(u_j,u_j)
       +\|u_j\|_{L^2(\operatorname{supp}d\chi_p)}^2\bigr)
 \longrightarrow0.
$$
Approximation by compactly supported sections extends these estimates
to $\operatorname{Dom}(\mathfrak q_{\lambda_j,q_j,\delta_j})$.
The disks and the exterior cover $X$, contradicting normalization.
\end{proof}

For $f\in |w|^{\widehat\eta-2}C_b^{0,\alpha}$, we solve
$\mathcal Lu=f$ first in $\operatorname{Dom}(\mathfrak q)$.
Choose $1<p<2/(2-\widehat\eta)$.  Since
$$
 \int_0^{\varrho}r^{p(\widehat\eta-2)+1}\,dr<\infty,
 \qquad \|f\|_{L^p}\leq C\|f\|_{|w|^{\widehat\eta-2}C_b^{0,\alpha}},
$$
Kato's inequality and the scalar two-dimensional Sobolev embedding
give, for $v\in\operatorname{Dom}(\mathfrak q)$,
$$
 |\langle f,v\rangle|
 \leq C\|f\|_{|w|^{\widehat\eta-2}C_b^{0,\alpha}}
       \bigl(\|d|v|\|_{L^2}+\|v\|_{L^2}\bigr)
 \leq C\|f\|_{|w|^{\widehat\eta-2}C_b^{0,\alpha}}\mathfrak q(v,v)^{1/2}.
$$
The scalar embeddings hold for $g_c$ by lifting to its finite disk
covers and using a fixed partition of unity.  The Lax--Milgram theorem
gives a unique $u\in\operatorname{Dom}(\mathfrak q)$ with
$\mathfrak q(u,v)=\langle f,v\rangle$ for every
$v\in\operatorname{Dom}(\mathfrak q)$, and
\begin{equation}\label{eq:weak-solution-energy}
 \mathfrak q(u,u)^{1/2}+\|u\|_{L^2}
 \leq C\|f\|_{|w|^{\widehat\eta-2}C_b^{0,\alpha}}.
\end{equation}
To use this solution in the nonlinear equation, we control its
constant diagonal term and its weighted remainder by identifying the
indicial operator at each puncture.

Near a puncture, we put $\widehat w=3\lambda w_p$ and
$\rho=|\widehat w|$, and abbreviate $u_j=u_j^{\delta_p}$, where
$\delta_0=\delta$ and $\delta_1=\delta_\infty=0$.
In the cyclic holomorphic frame $f_j$, the model metric has entries
$e^{2u_j}$.  Its unitary frame is $e^{-u_j}f_j$.  Transforming the
Chern connection and the Higgs field to this frame gives
$$
 A_{\delta,p}^{\rm mod}
   =\operatorname{diag}(\partial u_j-\bar\partial u_j)
   =\ii\operatorname{diag}(\rho u_j')\,d\vartheta,
 \qquad
 (\Phi_{\delta,p}^{\rm mod})_{ij}
   =P_{ij}e^{u_i-u_j}\,d\widehat w.
$$
In particular, the connection has no radial component.  Proposition
\ref{prop:toda-decay} implies
$$
 \rho u_j'=\gamma_j^{\delta_p}+O_b(\rho^{b_I}),\qquad
 \rho^2|\Phi_{\delta,p}^{\rm mod}|^2
   =O_b(\rho^{2-6\delta_p})=O_b(\rho^{b_I}).
$$
These bounds include every fixed order of $b$-derivatives and are
uniform on $I$.
Multiplying the linearization by $\rho^2$ therefore gives
\begin{equation}\label{eq:indicial-operator}
 \begin{aligned}
 \rho^2\mathcal L_{\delta,p}^{\rm mod}
 &=-(\rho\partial_\rho)^2+B_{\delta,p}^2
   +\rho^{b_I}R_{\delta,p},\\
 B_{\delta,p}&=-\ii\partial_\vartheta+
   \operatorname{ad}\diag(-\delta_p,-\delta_p,2\delta_p),
 \end{aligned}
\end{equation}
where $R_{\delta,p}$ has order at most one in the $b$-vector fields
and uniformly bounded coefficients with their $b$-derivatives.
The scalar principal part is independent of the parameters.

The deck transformation $w_p\mapsto\omega w_p$ multiplies $f_j$,
and hence its unitary normalization, by $\omega^{1-j}$.
Consequently an endomorphism matrix $v$ descends precisely when
$$
 v_{ij}(\rho,\vartheta+2\pi/3)
   =\omega^{i-j}v_{ij}(\rho,\vartheta).
$$
Thus a Fourier mode $e^{\ii n\vartheta}E_{ij}$ is allowed when
$n\equiv i-j\pmod3$.  Substitution of
$\rho^\sigma e^{\ii n\vartheta}E_{ij}$ gives the indicial equations
\begin{equation}\label{eq:indicial-roots}
 \sigma^2=(n+\gamma_i^{\delta_p}-\gamma_j^{\delta_p})^2,
 \qquad
 (\gamma_0^{\delta_p},\gamma_1^{\delta_p},\gamma_2^{\delta_p})
 =(-\delta_p,-\delta_p,2\delta_p).
\end{equation}
On diagonal entries the allowed frequencies are $3\Z$; on
off-diagonal entries they lie in $3\Z\pm1$, and the connection
shifts have absolute value at most $3\delta_p$.  The indicial
solutions associated with the root zero are $c+d\log\rho$, with
$c,d$ constant trace-free diagonal Hermitian matrices.  Every nonzero indicial root has absolute
value at least $1-3\sup I$.  The weight $\widehat\eta$ in $D_{\rm Fr}$ lies below the first positive
indicial root by a positive amount independent of $\delta\in I$.  We now use this gap to prove that the weak solution
belongs to $D_{\rm Fr}$.

\begin{thm}
\label{thm:uniform-inverse}
For $T\geq T_0$, the weak inverse defines an isomorphism
$$
 G_{\lambda,q,\delta}:
 |w|^{\widehat\eta-2}C_b^{0,\alpha}
 \longrightarrow
 D_{\rm Fr}.
$$
There are $C,N>0$, uniform for
$(q,\delta,\arg\lambda)\in \Pic^0(Y)\times I\times S^1$, such that
\begin{equation}\label{eq:uniform-inverse}
 \|G_{\lambda,q,\delta}f\|_{D_{\rm Fr}}
 \leq CT^N\|f\|_{|w|^{\widehat\eta-2}C_b^{0,\alpha}}.
\end{equation}
\end{thm}

\begin{proof}
Let $u=G_{\lambda,q,\delta}f$,
$F=\|f\|_{|w|^{\widehat\eta-2}C_b^{0,\alpha}}$, and $a=3T$.

\emph{Rescaling and the weak estimate.}
Choose a fixed small $\rho_0>0$.  For large $T$, the approximate
metric agrees with the local model on
$|w_p|<4\rho_0/a$.  In the normalized coordinate, the equation becomes
$\mathcal L_{\delta,p}^{\rm mod}\widehat u=\widehat f$, where
$\widehat u(\widehat w)=u(w_p)$ and
$\widehat f(\widehat w)=a^{-2}f(w_p)$.  Hence,
$$
 \|\widehat f\|_{\rho^{\widehat\eta-2}C_b^{0,\alpha}}
 \leq Ca^{-\widehat\eta}F,
 \qquad
 \|\widehat u\|_{L^2(\rho<4\rho_0)}
 \leq Ca\|u\|_{L^2}.
$$
The covariant Dirichlet integral is invariant under this change of
scale.  Kato's inequality and the scalar Sobolev embedding on the
fixed disk therefore give, for every $1<\ell<\infty$,
$$
 \|\widehat u\|_{L^\ell(\rho<3\rho_0)}
 \leq C_\ell\bigl(\mathfrak q_{\lambda,q,\delta}(u,u)^{1/2}
                    +a\|u\|_{L^2}\bigr).
$$
The scalar norm extends across the center as a $W^{1,2}$ function:
this follows by applying Kato's inequality to the compactly supported
approximants converging in the norm of
$\operatorname{Dom}(\mathfrak q_{\lambda,q,\delta})$.

\emph{An initial negative weight.}
Choose
$0<\epsilon<b_I-\widehat\eta$ and
$\ell>2/\epsilon$.  For $0<r<\rho_0/2$, rescale the annulus
$r<\rho<2r$ and its enlargement $r/2<\rho<3r$ to fixed
annuli.  The coefficients of the rescaled model are uniformly
bounded, so interior elliptic and Schauder estimates give
$$
 \|\widehat u\|_{C_b^{2,\alpha}(r<\rho<2r)}
 \leq C\left(a^{-\widehat\eta}r^{\widehat\eta}F
       +r^{-2/\ell}\|\widehat u\|_{L^\ell(\rho<3\rho_0)}\right).
$$
It follows that
$$
 \|\widehat u\|_{\rho^{-\epsilon}C_b^{2,\alpha}(\rho<\rho_0)}
 \leq C\left(\mathfrak q_{\lambda,q,\delta}(u,u)^{1/2}
                    +a\|u\|_{L^2}+a^{-\widehat\eta}F\right).
$$

\emph{The constant diagonal term and the weighted remainder.}
Denote the indicial operator by
$P^{\rm ind}=-(\rho\partial_\rho)^2+B_{\delta,p}^2$.
The equation and $b_I-\epsilon>\widehat\eta$ show that
$$
 P^{\rm ind}\widehat u
 =\rho^2\widehat f-\rho^{b_I}R_{\delta,p}\widehat u
 \in\rho^{\widehat\eta}C_b^{0,\alpha}.
$$
Insert a cutoff $\chi$ supported in $\rho<\rho_0$ and equal to one
on $\rho<\rho_0/2$.  Its commutator is supported in a fixed annulus
and is controlled by the preceding estimate.  Put $s=\log\rho$ and
$g=P^{\rm ind}(\chi\widehat u)$.  The required inverse on the
cylinder can be described explicitly.  If $g_\mu$ is the component
with angular eigenvalue $\mu\ne0$, set
$$
 v_\mu(s)=\frac1{2|\mu|}
   \int_{\R}e^{-|\mu||s-t|}g_\mu(t)\,dt.
$$
Since $|\mu|\geq1-3\sup I>\widehat\eta$, its weighted bound is
$$
 \|e^{-\widehat\eta s}v_\mu\|_{L^\infty}
 \leq\frac{\|e^{-\widehat\eta s}g_\mu\|_{L^\infty}}
                  {\mu^2-\widehat\eta^2}.
$$
On the constant diagonal component the corresponding solution is
$$
 v_0(s)=-\int_{-\infty}^s(s-t)g_0(t)\,dt,\qquad
 \|e^{-\widehat\eta s}v_0\|_{L^\infty}
 \leq\widehat\eta^{-2}
              \|e^{-\widehat\eta s}g_0\|_{L^\infty}.
$$
The nonzero eigenvalues grow linearly with the Fourier frequency,
so the preceding inverse bounds are summable, uniformly in $\delta$.
They give a weighted $C^0$ bound for $v$.  Applying interior Schauder
estimates to $e^{-\widehat\eta s}v$ on unit cylinder strips then gives
$$
 \|v\|_{\rho^{\widehat\eta}C_b^{2,\alpha}}
 \leq C\|g\|_{\rho^{\widehat\eta}C_b^{0,\alpha}}.
$$
These Green operators realize the indicial inverse with Mellin
multiplier $(\mu^2-\sigma^2)^{-1}$ on
$\Re\sigma=\widehat\eta$; compare \cite{Maz91}.
For the homogeneous difference $\chi\widehat u-v$, the initial
weight $-\epsilon$ excludes negative nonzero branches at
$s=-\infty$, while the cutoff and the Green formulas exclude
positive nonzero branches at $s=+\infty$.  The remaining terms are
constant and logarithmic diagonal terms.  Consequently
$$
 \widehat u=c_p+d_p\log\rho+v_p,
 \qquad v_p\in\rho^{\widehat\eta}C_b^{2,\alpha},
$$
where $c_p,d_p$ are constant trace-free diagonal Hermitian matrices.
The coefficient and remainder norms are bounded by
$C(\mathfrak q_{\lambda,q,\delta}(u,u)^{1/2}
+a\|u\|_{L^2}+a^{-\widehat\eta}F)$.
The coefficient bound follows by evaluating the homogeneous difference
and its first radial derivative on a fixed annulus.  The logarithmic
term contributes $|d_p|^2\int_0^{\rho_0/2}\rho^{-1}\,d\rho$ to the
Dirichlet integral, while the remainder has finite energy.  Thus
$d_p=0$.  Scaling back multiplies the weighted remainder norm by
$a^{\widehat\eta}$ and leaves $c_p$ unchanged.  Using
\eqref{eq:weak-solution-energy} yields
$$
 |c_p|+
 \|u-c_p\|_{|w_p|^{\widehat\eta}C_b^{2,\alpha}
                (|w_p|<\rho_0/(2a))}
 \leq Ca^{1+\widehat\eta}F.
$$

\emph{The remaining annuli and the compact interior.}
It remains to estimate $u$ outside these disks.  On an annulus
$A_r=\{r<|w_p|<2r\}$, with $r\geq\rho_0/(4a)$, rescale
$w_p=r\zeta$.  The principal part of $r^2\mathcal L$ is the
fixed scalar Laplacian; the lower-order coefficient norms needed for
Schauder estimates grow at most polynomially in $T$.  This follows
from the uniform model bounds for $a|w_p|\geq\rho_0/8$ and the
fixed cutoffs defining $h^{\rm app}$.  To track this dependence, put
$U(\zeta)=u(r\zeta)$.  On fixed nested annuli the equation has the
form
$$
 -\Delta_\zeta U+\sum_{j=1}^2 a_j\partial_jU+a_0U
      =r^2f(r\zeta),\qquad
 \sum_{j=0}^2\|a_j\|_{C^{0,\alpha}}\leq CT^{m_0},
$$
for one fixed $m_0$.  The coefficients act on the endomorphism
components.  On each annulus, interpolation gives
$$
 \|U\|_{C^{1,\alpha}}
 \leq\varepsilon\|U\|_{C^{2,\alpha}}
       +C\varepsilon^{-k_0}\|U\|_{L^2},\qquad 0<\varepsilon<1,
$$
with a fixed finite $k_0$.  On the rescaled coordinate plane, let
$A_s=\{1-s<|\zeta|<2+s\}$ for $0\leq s\leq1/4$.
Combining interpolation with the interior estimate for the scalar
Laplacian gives, for $0\leq s<t\leq1/4$,
$$
 \|U\|_{C^{2,\alpha}(A_s)}
 \leq\theta\|U\|_{C^{2,\alpha}(A_t)}
   +CT^M(t-s)^{-L}
    \bigl(\|r^2f(r\,\cdot)\|_{C^{0,\alpha}(A_{1/4})}
                         +\|U\|_{L^2(A_{1/4})}\bigr).
$$
Here $M,L$ are fixed, and the interpolation parameter is chosen as a
small multiple of $T^{-m_0}$ times a power of $t-s$, so that
$0<\theta<2^{-L}$.  Iterating with
$s_j=\tfrac14(1-2^{-j})$ absorbs the first term because
$\sum_j\theta^j2^{jL}<\infty$.  The closure of $A_{1/4}$ lies inside
$\{1/2<|\zeta|<4\}$, so the terminal norm is finite by interior
regularity.  The resulting constants are polynomial
in $T$.  Finally,
$\|r^2f(r\,\cdot)\|_{C^{0,\alpha}}\leq Cr^{\widehat\eta}F$
and $\|U\|_{L^2}=r^{-1}\|u\|_{L^2}$.  We obtain
$$
 \|u\|_{C_b^{2,\alpha}(A_r)}
 \leq CT^M\left(r^{\widehat\eta}F
       +r^{-1}\|u\|_{L^2(A'_r)}\right),
$$
where $A'_r=\{r/2<|w_p|<4r\}$.
The same argument applies on the fixed compact region.  Since
$r^{-1-\widehat\eta}\leq Ca^{1+\widehat\eta}$, the energy
estimate bounds the weighted remainder on these annuli by a finite
power of $T$ times $F$.  Subtracting $\psi_pc_p$ has the same
bound.  Taking the supremum over the annuli and combining it with the
estimates on the inner disks proves the $D_{\rm Fr}$ estimate
\eqref{eq:uniform-inverse}.  Uniqueness in
$\operatorname{Dom}(\mathfrak q_{\lambda,q,\delta})$ proves
injectivity of \eqref{eq:linearized-holder-map}; the weak construction
and weighted estimates prove surjectivity.  All constructions preserve
the deck-equivariance condition, and their constants are uniform on
the compact parameter sets.
\end{proof}

\subsection{Correction to a harmonic metric}

The norm \eqref{eq:holder-domain-norm} on $D_{\rm Fr}$ controls the
constant diagonal terms $c_p$, the
$|w_p|^{\widehat\eta}C_b^{2,\alpha}$ remainders, and the
$L^\infty$ norm.  We solve the nonlinear equation on a small ball
in $D_{\rm Fr}$, with radius determined by the norm of the residual
in $|w|^{\widehat\eta-2}C_b^{0,\alpha}$.
For fixed parameters set
$$
 Q(\gamma)=\mathcal F(\gamma)-R-\mathcal L\gamma.
$$
The local formula for the curvature of $h^{\rm app}e^{2\gamma}$
has the structure described in \cite[Proposition~6.1 and
equation~(6.1)]{FMSW22}.  The following estimate records the weights
needed for our cyclic model.

\begin{lem}\label{lem:nonlinear-mapping}
There are $r_*,C,M>0$, uniform on $\Pic^0(Y)\times I\times S^1$, such that
\begin{equation}\label{eq:nonlinear-estimate}
 \begin{aligned}
 &\|Q(\gamma_1)-Q(\gamma_2)\|_{|w|^{\widehat\eta-2}C_b^{0,\alpha}}\\
 &\qquad\leq CT^M
  (\|\gamma_1\|_{D_{\rm Fr}}+\|\gamma_2\|_{D_{\rm Fr}})
  \|\gamma_1-\gamma_2\|_{D_{\rm Fr}}
 \end{aligned}
\end{equation}
for $T\geq1$ and $\gamma_j\in D_{\rm Fr}$ with
$\|\gamma_j\|_{D_{\rm Fr}}\leq r_*$.  On this ball,
$$
 Q:\{\gamma\in D_{\rm Fr}:\|\gamma\|_{D_{\rm Fr}}<r_*\}
       \longrightarrow |w|^{\widehat\eta-2}C_b^{0,\alpha}
$$
satisfies $Q(0)=DQ(0)=0$.  In the parameter identifications constructed
below, this map is smooth in $(\lambda,q,\gamma)$ for fixed
$\delta$ and continuous in $\delta$ in these Banach norms.
\end{lem}

\begin{proof}
The $D_{\rm Fr}$ norm controls $\|\gamma\|_{L^\infty}$ uniformly, so the
matrix exponential and its derivatives are uniformly bounded on a
fixed ball.  After the constant and linear terms have been removed,
the differentiated terms of $Q$ have the form of a coefficient
vanishing at zero times $\nabla_A^2\gamma$, or a smooth coefficient
times $(\nabla_A\gamma)^2$.  Near a puncture write $\gamma=c+\widetilde\gamma$.
The connection is diagonal and commutes with $c$; the differentiated
terms belong to $|w_p|^{\widehat\eta-2}C_b^{0,\alpha}$ and
$|w_p|^{2\widehat\eta-2}C_b^{0,\alpha}$, respectively.  The second
space embeds continuously in the first because $\widehat\eta>0$.

The terms involving the Higgs field are at least quadratic in
$\gamma$.  At $p=0$ the Toda asymptotics give, for $\rho\leq1$,
$|\Phi|_{h^{\rm app}}^2\leq CT^2\rho^{-6\delta}$, and for
$\rho\geq1$ the metric is uniformly equivalent to the decoupled
metric.  Thus the most singular contribution from the constant
diagonal terms is bounded by
$$
 C\bigl(T^2+T^{2-6\delta}|w|^{-6\delta}\bigr)
 (|c_1|+|c_2|)|c_1-c_2|.
$$
Multiplication by $|w|^{2-\widehat\eta}$ leaves the positive exponent
$2-6\delta-\widehat\eta$, uniformly bounded away from zero on $I$.
At $p=1,\infty$ the logarithmic exponents vanish.  Terms containing
the remainder gain the positive weight $\widehat\eta$.
Weighted H\"older multiplication and the mean-value formula now give
the estimate.  On the compact interior ordinary multiplication gives
the same conclusion.  All background coefficient bounds have at most
polynomial growth in $T$.

We compare $D_{\rm Fr}$ and $|w|^{\widehat\eta-2}C_b^{0,\alpha}$
for nearby parameters by choosing determinant-preserving smooth
identifications of the bundles $E_q$ over a small neighborhood in
$\Pic^0(Y)$.  The smooth family $\mathscr E$ admits such
identifications.  We choose them to identify the cyclic holomorphic
frames near each puncture and extend them over the compact interior.
We then compose with the inverse square root of the relative metric
endomorphism.  More explicitly, for two compared metrics
$h_2=h_1e^{2s}$, the map $e^{-s}$ is an isometry from $h_1$ to $h_2$;
both metrics have determinant one, so $\tr s=0$.
In the associated unitary frames the identifications preserve constant
diagonal matrices and the spaces
$|w_p|^{\widehat\eta}C_b^{2,\alpha}$ and
$|w_p|^{\widehat\eta-2}C_b^{0,\alpha}$.
The coefficients depend smoothly on $(\lambda,q)$ for fixed $\delta$
by the spatial Toda estimates and the smooth Hermitian metric on
$\mathscr P$.  They are continuous in $\delta$ in the operator norms from
$D_{\rm Fr}$ to $|w|^{\widehat\eta-2}C_b^{0,\alpha}$: the positive margins in the conormal remainder and
in $2-6\sup I-\widehat\eta$ absorb the variation of powers of $|w|$.
For example, if $\beta_1,\beta_2\geq\beta_*>0$, then
$$
 \sup_{0<r\leq1}|r^{\beta_1}-r^{\beta_2}|
 \leq |\beta_1-\beta_2|
             \sup_{0<r\leq1}r^{\beta_*}|\log r|
 =\frac{|\beta_1-\beta_2|}{e\beta_*}.
$$
The same argument applies after any fixed number of $b$-derivatives;
one additional derivative controls the local H\"older seminorm.
On the compact interior this is ordinary smooth parameter dependence.
The local formulas for $\mathcal F$ and weighted multiplication prove
the asserted regularity.
\end{proof}

We now apply the contraction mapping theorem.
\begin{prop}\label{prop:correction}
For $T$ sufficiently large there is a unique normalized adapted
harmonic metric
\begin{equation}\label{eq:harmonic-correction}
 \begin{aligned}
 h_{\lambda,q,\delta}
 &=h_{\lambda,q,\delta}^{\rm app}e^{2\gamma_{\lambda,q,\delta}},\\
 \|\gamma_{\lambda,q,\delta}\|_{D_{\rm Fr}}
 &\leq C_Ie^{-c_IT}.
 \end{aligned}
\end{equation}
In the parameter charts constructed in the proof of
Lemma~\ref{lem:nonlinear-mapping}, the correction depends smoothly on
$(\lambda,q)$ for fixed $\delta$ and continuously on $\delta$ in the
$D_{\rm Fr}$-norm.  It is compatible with the spectral quotient.
The estimate is uniform on $\Pic^0(Y)\times I\times S^1$.
\end{prop}

\begin{proof}
The operator norm of
$G:|w|^{\widehat\eta-2}C_b^{0,\alpha}\to D_{\rm Fr}$ and
the residual satisfy
$$
 \|G\|\leq C_GT^N,\qquad
 \|R\|_{|w|^{\widehat\eta-2}C_b^{0,\alpha}}\leq C_Re^{-c_RT},
$$
We set $r_T=2C_GC_RT^Ne^{-c_RT}$.  For large $T$, the map
$$
 \gamma\longmapsto-G\bigl(R+Q(\gamma)\bigr)
$$
is defined on
$\{\gamma\in D_{\rm Fr}:\|\gamma\|_{D_{\rm Fr}}\leq r_T\}$.
The preceding lemma gives
$$
 \|G(R+Q(\gamma))\|_{D_{\rm Fr}}
 \leq\frac{r_T}{2}+C_GCT^{N+M}r_T^2
 \leq\frac{3r_T}{4},
$$
after increasing $T_0$.  Its Lipschitz constant is at most
$2C_GCT^{N+M}r_T<1/2$.  Thus it preserves this ball.
It has a unique fixed point, and any $0<c_I<c_R$ absorbs the
polynomial in its norm bound.  At that point
$$
 D_\gamma\mathcal F
 =\mathcal L\bigl(\Id+GDQ(\gamma)\bigr)
$$
is invertible from $D_{\rm Fr}$ to
$|w|^{\widehat\eta-2}C_b^{0,\alpha}$ by the same estimate.
The implicit function theorem, using the parameter identifications of
Lemma~\ref{lem:nonlinear-mapping}, gives smoothness in $(\lambda,q)$;
continuous dependence of the contraction gives continuity in
$\delta$.

The metric $h=h^{\rm app}e^{2\gamma}$ is mutually bounded with
$h^{\rm app}$ and hence adapted to the same filtration.  It has
determinant one, and $\mathcal F(\gamma)=0$ is Hitchin's equation.
Interior elliptic regularity makes it smooth on $X^\circ$.
Stability from Proposition~\ref{prop:spectral} and uniqueness of the
normalized adapted harmonic metric
\cite[Proposition~5.2]{Mochizuki} identify it with the metric defining
the Hitchin $L^2$ metric.  This also identifies the constructions on
parameter overlaps.  The decoupled metric, radial Toda interpolation,
and cutoffs are compatible with the isomorphisms of
Proposition~\ref{prop:spectral-quotient}; uniqueness gives the same
compatibility for $h$ and the self-adjoint endomorphism $\gamma$.
\end{proof}

This proves Theorem~\ref{thm:intro-construction}.  The metric comparison
will use the uniform $L^\infty$ bound on $\gamma$, the residual bound,
and the form coercivity established in this section.

\section{Comparison of the Hitchin and semi-flat metrics}
\label{sec:metric-comparison}

Fix $I\Subset(0,1/3)$.  We construct representatives of the horizontal
and vertical tangent classes, compare their pairings, and then project
them to the harmonic representatives for the exact metric.  This follows
the method of \cite[Section~4]{Mochizuki23}; the estimates below justify
the integrations by parts at the parabolic points.  We abbreviate
$$
 \Phi_T=\Phi_{\lambda,q},\qquad
 h_T^{\rm app}=h_{\lambda,q,\delta}^{\rm app},\qquad
 h_T=h_{\lambda,q,\delta},\qquad
 \gamma_T=\gamma_{\lambda,q,\delta}.
$$
Proposition~\ref{prop:correction} identifies $h_T$ with the normalized
harmonic metric defining $g_{L^2,\delta}$.  Constants are uniform for
$\delta\in I$ and $(\arg\lambda,q)\in S^1\times\Pic^0(Y)$.

On the fixed elliptic spectral curve $\pi:Y\to X$, write $M_f$ for
direct-image multiplication by a scalar form $f$ on $Y$.
On the unramified locus,
$$
 \pi^*M_f=\diag(f,\sigma^*f,\sigma^{2*}f),
 \qquad
 \mathbf F_\nu=(0,M_\nu),\qquad
 \mathbf F_\tau=(M_\tau,0).
$$
Here $\nu\in H^0(Y,K_Y)$ and $\tau$ is a harmonic $(0,1)$-form.
Let $\theta$ denote the tautological differential on the spectral
curve.  The spectral identifications $f_\lambda$ trivialize the homology local
system.  Holding $q\in\Pic^0(Y)$ fixed therefore
gives the Gauss--Manin horizontal lift, and differentiating
$f_\lambda^*\theta=\lambda\alpha$ gives
$\nu=\dot\lambda\alpha$.  At fixed $\lambda$, differentiation of the
family $L_q=L_0\otimes P_q$ gives the vertical class $\tau$.  Thus the product
splitting under $d\widetilde\Psi_\delta$ is the Gauss--Manin
horizontal--vertical splitting.  The corresponding deformation classes
are represented on the unramified locus by $\mathbf F_\nu$ and
$\mathbf F_\tau$.  Both have vanishing direct-image trace.

For a gauge section $s$ and a deformation pair $\mathbf x=(a_x,\varphi_x)$,
write
$$
 D_T''s=(\bar\partial_Es,[\Phi_T,s]),\qquad
 D_T''\mathbf x=\bar\partial_E\varphi_x+[a_x,\Phi_T].
$$
For a bundle metric $h$, the deformation pairing is
\begin{equation}\label{eq:deformation-pairing}
 \langle\mathbf x,\mathbf y\rangle_h
 =2\ii\int_X\tr\left(
 \varphi_x\wedge\varphi_y^{\dagger_h}
 -a_x\wedge a_y^{\dagger_h}\right).
\end{equation}
The subscripts $\mathrm{app}$ and $\mathrm{ex}$ indicate the metrics
$h_T^{\rm app}$ and $h_T$, respectively, and
$$
 Q_\bullet(s,r)=\langle D_T''s,D_T''r\rangle_\bullet,
 \qquad \bullet\in\{\mathrm{app},\mathrm{ex}\}.
$$
The pairing is complex linear in its first argument.  The components
of a deformation are always ordered as $(a,\varphi)$, with types
$(0,1)$ and $(1,0)$, respectively.  The real part of
\eqref{eq:deformation-pairing} is the $L^2$ inner product on these
pairs.  For $h=h_T$, its restriction to harmonic representatives
gives $g_{L^2,\delta}$.  The operator $D_T''$ depends on the
holomorphic Higgs bundle, while its formal adjoint and the pairing
depend on the chosen bundle metric.
Let $\mathcal Q$ be the completion of
$C_c^\infty(X\setminus D,\operatorname{End}_0E)$ in the norm
\begin{equation}\label{eq:coulomb-form-domain}
 \|r\|_{\mathcal Q,{\rm app}}^2
 =\|r\|_{L^2,{\rm app}}^2+Q_{\rm app}(r,r).
\end{equation}
The completion $\mathcal Q$ is a Hilbert space of trace-free complex
endomorphisms, and $D_T''$ maps it continuously to
$$
 L^2\bigl(X^\circ,
  (\Lambda^{0,1}\oplus\Lambda^{1,0})\otimes\operatorname{End}_0E\bigr),
$$
with the metric $h_T^{\rm app}$ and the pairing
\eqref{eq:deformation-pairing}.
A trace-free section $u$ with bounded $h_T^{\rm app}$-norm,
local $W^{1,2}$ regularity on $X^\circ$, and $D_T''u\in L^2$
belongs to $\mathcal Q$.  Indeed, logarithmic cutoffs $\eta_\epsilon$, equal to
zero for $r\leq\epsilon^2$ and one for $r\geq\epsilon$, satisfy
$|d\eta_\epsilon|\leq C/(r|\log\epsilon|)$ and
$$
 \|(\bar\partial\eta_\epsilon)u\|_{L^2}^2
 \leq\frac{C\|u\|_{L^\infty}^2}{|\log\epsilon|}\longrightarrow0.
$$
The other integral tails tend to zero, using the finite conic volume
for zero-forms and conformal invariance for one-forms.  Approximation
on compact subsets proves the assertion.

We use the norms
$$
 \|\nu\|_Y^2=2\ii\int_Y\nu\wedge\bar\nu,
 \qquad
 \|\tau\|_Y^2=-2\ii\int_Y\tau\wedge\bar\tau.
$$
Choose a fixed $\varrho_0>0$ such that $\chi_0=1$ on
$\{|w_0|\leq\varrho_0\}$.  On this disk,
$$
 h_{0,T}=\diag(e^{2u_i(3T|w_0|)}),\qquad
 \Phi_T=3\lambda P\,dw_0.
$$
At $p=1,\infty$, let $h_{p,T}$ be the decoupled local metric.
In the local calculations, write
$$
 P_i=E_{i+1,i},\qquad P=\sum_iP_i,\qquad
 d_i(\rho)=u_{i+1}(\rho)-u_i(\rho),
 \qquad i\in\Z/3\Z.
$$
On the $w_0$-disk write $A_{\epsilon,R}=\{\epsilon<|w_0|<R\}$.
The local fiberwise adjoint for $h_{p,T}$ is denoted by $\dagger_T$;
formal adjoints also use $g_c$.  A pairing restricted to a local model
uses $h_{p,T}$ and the normalization above.  The boundary
term in the Green formula for $D_T''$ is
$$
 \mathfrak d_r^T(s,\mathbf y)
 =-\frac{2\ii}{3}\int_{|w_p|=r}\tr(s\,a_y^{\dagger_T}).
$$
The factor $1/3$ converts integration on the cubic cover to integration
on the base.  All local pairings below include this covering factor.  With both
circles oriented counterclockwise, Green's formula gives
\begin{equation}\label{eq:local-green-formula}
 \langle D_T''s,\mathbf y\rangle_{A_{\epsilon,R}}
 =\langle s,(D_T'')^{*,T}\mathbf y\rangle_{A_{\epsilon,R}}
  +\mathfrak d_R^T(s,\mathbf y)-\mathfrak d_\epsilon^T(s,\mathbf y).
\end{equation}
We first apply this formula with $\epsilon>0$.  The local estimates
will justify the limit as $\epsilon$ tends to zero.

\subsection{Horizontal representatives}

Set $\lambda=Te^{\ii\phi}$.  Since $Y$ is elliptic, every horizontal
form can be written as $\nu_\xi=\xi e^{\ii\phi}\alpha$.
On the Toda disk put
\begin{equation}\label{eq:horizontal-local-representative}
 \begin{gathered}
 \rho=3T|w_0|,\qquad
 s_\xi=\frac{\xi}{T}\rho\diag(u_i'(\rho)),\\
 \mathbf H_0(\nu_\xi,T)=\mathbf F_{\nu_\xi}-D_T''s_\xi.
 \end{gathered}
\end{equation}
Thus its two components are
$$
 \begin{aligned}
 a_{\mathbf H_0}&=-\bar\partial_Es_\xi,\\
 \varphi_{\mathbf H_0}
 &=M_{\nu_\xi}+[s_\xi,\Phi_T]
 =3e^{\ii\phi}\xi\sum_i(1+\rho d_i')P_i\,dw_0.
 \end{aligned}
$$
The tangent parameter satisfies $\dot T=\Re\xi$ and
$T\dot\phi=\Im\xi$.  The formula for $\mathbf H_0$ is complex linear
in $\xi$, as required by the complex structure of the Higgs moduli
space.
We extend this representative by
\begin{equation}\label{eq:horizontal-representative}
 \mathbf H'(\nu_\xi,T)=\mathbf F_{\nu_\xi}-D_T''(\chi_0s_\xi).
\end{equation}

\begin{prop}
\label{lem:horizontal-representative}
The pair $\mathbf H'(\nu,T)$ is closed and trace-free, and
$-\chi_0s_\xi\in\mathcal Q$.  For $T$ sufficiently large,
$$
 \left|
 \langle\mathbf H'(\nu_1,T),\mathbf H'(\nu_2,T)\rangle_{\rm app}
 -2\ii\int_Y\nu_1\wedge\overline{\nu_2}
 \right|
 \leq Ce^{-cT}\|\nu_1\|_Y\|\nu_2\|_Y.
$$
In particular, $\|\mathbf H'(\nu,T)\|_{\rm app}\leq C\|\nu\|_Y$.
\end{prop}

\begin{proof}
Since $D_T''\mathbf F_{\nu_\xi}=0$ and $(D_T'')^2=0$, the pair
$\mathbf H_0$ is closed.  Differentiating the Toda equation gives
$$
 \begin{aligned}
 (\rho u_i')''+\rho^{-1}(\rho u_i')'
 =4\bigl[&e^{2d_{i-1}}(1+\rho d_{i-1}')\\
         &-e^{2d_i}(1+\rho d_i')\bigr].
 \end{aligned}
$$
Using $\rho=3T|w_0|$, we obtain
\begin{equation}\label{eq:horizontal-coulomb-equation}
 \begin{aligned}
 (\partial_{E,h_{0,T}}\bar\partial_Es_\xi)_{ii}
 &=9T\xi\bigl[e^{2d_{i-1}}(1+\rho d_{i-1}')
              -e^{2d_i}(1+\rho d_i')\bigr]
              dw_0\wedge d\bar w_0\\
 &=[\Phi_T^{\dagger_T},\varphi_{\mathbf H_0}]_{ii}.
 \end{aligned}
\end{equation}
Both sides are diagonal, so this proves the gauge equation in
\cite[equation~(7.4)]{FMSW22}.  The K\"ahler identities then give
$(D_T'')^{*,T}\mathbf H_0=0$.  The determinant-one condition gives
$\tr s_\xi=0$.  Since
$\bar\partial\rho=(\rho/(2\bar w_0))d\bar w_0$, the Toda equation gives
$$
 \bar\partial_Es_{\xi,i}
 =\frac{\xi}{T}\frac{\rho^2}{\bar w_0}
   \left(e^{2d_{i-1}}-e^{2d_i}\right)d\bar w_0.
$$
The induced norm on diagonal endomorphisms is the Euclidean norm.
The endpoint estimates therefore imply
$$
 |s_\xi|\leq C_I\frac{|\xi|}{T},\qquad
 |\bar\partial_Es_\xi|\leq C_I|\xi|\rho^{1-6\delta},
 \qquad 0<\rho\leq1.
$$
Since $|P_i|_{h_{0,T}}=e^{d_i}$ and $\rho d_i'$ is bounded,
the Higgs component satisfies
$|\varphi_{\mathbf H_0}|\leq C_I|\xi|\rho^{-3\delta}$ on this disk.
The two components are square-integrable because
$$
 \int_0^1\rho^{3-12\delta}\,d\rho<\infty,
 \qquad
 \int_0^1\rho^{1-6\delta}\,d\rho<\infty,
$$
with bounds uniform for $\delta\in I$.  Thus the pairings below
converge at the puncture.  Multiplying the bounds for $s_\xi$ and
$\bar\partial_Es_\eta$ by the circle length gives
$$
 |\mathfrak d_r^T(-s_\xi,\mathbf H_0(\nu_\eta,T))|
 \leq C_I|\xi\eta|T^{-6\delta}r^{2-6\delta}.
$$
Toda decay gives the same bound when $\rho\geq1$.
For $T\geq1$ and $r\leq1$ this is bounded by
$C_I|\xi\eta|r^{2-6\sup I}$, which tends to zero uniformly.
At a fixed outer circle, $s_\xi$ and its required derivatives are
$O(e^{-cT})\|\nu_\xi\|_Y$ by Proposition~\ref{prop:toda-decay}.
In particular, the Coulomb condition and Hermitian symmetry give
$$
 \begin{aligned}
 &\langle\mathbf H_0(\nu_\xi,T),D_T''(-s_\eta)\rangle_{A_{\epsilon,R}}\\
 &\quad=\overline{\mathfrak d_R^T(-s_\eta,\mathbf H_0(\nu_\xi,T))
             -\mathfrak d_\epsilon^T(-s_\eta,\mathbf H_0(\nu_\xi,T))}.
 \end{aligned}
$$
Taking $\epsilon\downarrow0$ and $R=\varrho_0$ replaces the second
horizontal representative by $\mathbf F_{\nu_\eta}$ with an
exponentially small error.

The remaining integral has an explicit primitive.  Put
$S(\rho)=\sum_i e^{2d_i(\rho)}$.  The matrix units $P_i$ are
orthogonal and have squared norm $e^{2d_i}$.
Using $|dw_0|^2=2$, $dA_{w_0}=\frac{\ii}{2}dw_0\wedge d\bar w_0$
and the covering degree, the polarized difference from the decoupled
pairing on the inner disk is
\begin{equation}\label{eq:horizontal-boundary-identity}
 \begin{gathered}
 \frac{8\pi}{3}\frac{\xi\bar\eta}{T^2}
 \int_0^{3T\varrho_0}
 \rho\left(S-3+\frac\rho2S'\right)d\rho,\\
 \rho\left(S-3+\frac\rho2S'\right)
 =\frac12\frac d{d\rho}\left[\rho^2(S-3)\right].
 \end{gathered}
\end{equation}
The lower boundary vanishes since
$\rho^2(S-3)=O(\rho^{2-6\delta})$; the upper boundary is exponentially
small by Proposition~\ref{prop:toda-decay}.  These bounds are uniform
for $\delta\in I$.

On the exterior region $U\subset X\setminus D$, the direct-image
trace sums the three sheets, so the pairing is
$2\ii\int_{\pi^{-1}(U)}\nu_1\wedge\overline{\nu_2}$.  On the fixed transition annulus, the model and the
decoupled metric, as well as the representatives, differ by
$O(T^Ne^{-cT})$ in the required norms.  Summing the regions and decreasing
$c$ proves the estimate.  It also proves the uniform norm bound.
Finally, $\mathbf F_\nu$ is square-integrable: its pointwise norm
with respect to $h_{0,T}$ and $g_c$ is $O(r^{-3\delta})$ near $p=0$.  Thus
$D_T''(\chi_0s_\xi)=\mathbf F_\nu-\mathbf H'\in L^2$, and the cutoff
criterion gives $\chi_0s_\xi\in\mathcal Q$.
\end{proof}

\subsection{Vertical representatives}

We first localize the spectral deformation away from the punctures,
then add the local harmonic potentials.
Since $Y$ is elliptic and $3\,dw_p=\alpha$, there is a constant
$b=b(\tau)$ such that, on each coordinate disk,
$$
 \tau=b\,d\bar w_p,\qquad \beta_p(\tau)=b\bar w_p.
$$
These primitives vanish at the ramification points and satisfy
$\sigma^*\beta_p=\omega^2\beta_p$, as does $\tau$.
For fixed nested disks and cutoffs, put
$$
 \tau'=\tau-\sum_{p\in D}
 \bar\partial\bigl((\pi^*\chi_p)\beta_p(\tau)\bigr).
$$
Then $\tau'$ vanishes on the inner disks, represents the same class as
$\tau$, and
$$
 \|\tau'\|_{L^2}+\sum_p\|\beta_p(\tau)\|_{C^1}
 \leq C\|\tau\|_Y.
$$
The direct-image trace is zero because
$\sum_{j=0}^2\sigma^{j*}\tau'=0$.  The family
$\bar\partial_{L_q}+\varepsilon\tau'$, for small
$\varepsilon\in\C$, is constant on the inner disks and has the
same tangent class as the family $L_q=L_0\otimes P_q$.  Its direct image is a
family of strongly parabolic Higgs bundles, and $\mathbf F_{\tau'}$
represents the vertical tangent class.

On each local model define
\begin{equation}\label{eq:vertical-local-representative}
 S_p(\tau,T)=M_{\overline{\beta_p(\tau)}}^{\dagger_T},
 \qquad V_p(\tau,T)=D_T''S_p(\tau,T).
\end{equation}
Put $t_\tau=\sum_p\chi_pS_p(\tau,T)$ and define
\begin{equation}\label{eq:vertical-representative}
 \begin{aligned}
 \mathbf V'(\tau,T)
 &=\mathbf F_{\tau'}+D_T''t_\tau\\
 &=\mathbf F_\tau+
 \sum_pD_T''\bigl(\chi_p(S_p(\tau,T)-M_{\beta_p(\tau)})\bigr).
 \end{aligned}
\end{equation}
The equality follows from $[M_{\beta_p},\Phi_T]=0$.
The first expression relates $\mathbf V'$ to the strongly parabolic
deformation $\mathbf F_{\tau'}$.  The second expresses its difference
from $\mathbf F_\tau$ on the transition annuli, where the adjoint
defined by $h_{p,T}$ approaches the decoupled adjoint exponentially.

\begin{prop}
\label{lem:vertical-representative}
The pair $\mathbf V'(\tau,T)$ is closed and trace-free, its potential
$t_\tau$ belongs to $\mathcal Q$, and
$$
 \|\mathbf V'(\tau,T)\|_{\rm app}\leq C_I\|\tau\|_Y.
$$
On each inner disk it equals the Coulomb representative $V_p(\tau,T)$.
\end{prop}

\begin{proof}
Since $\beta_p$ is antiholomorphic,
$S_p^{\dagger_T}=M_{\overline{\beta_p}}$ is holomorphic and commutes
with $\Phi_T$.  Thus $D_T''S_p^{\dagger_T}=0$ and, by taking adjoints,
$D_T'S_p=0$, where
$$
 D_T'=\partial_{E,h_{p,T}}+\ad\Phi_T^{\dagger_T}.
$$
Since $h_{p,T}$ solves Hitchin's equations,
$$
 D_T'D_T''+D_T''D_T'=0,\qquad (D_T')^2=(D_T'')^2=0.
$$
Together with $D_T'S_p=0$ and the K\"ahler identities, this gives
$D_T''V_p=0$ and $(D_T'')^{*,T}V_p=0$; compare
\cite[Proposition~2.26]{Mochizuki23}.  These identities hold on every
annulus avoiding the puncture.

We check the boundary and integrability conditions at $p=0$.  Write
$r=|w_0|$ and $\rho=3Tr$.  In the cyclic frame,
$$
 \begin{aligned}
 S_0&=b\bar w_0\sum_i e^{2d_i}P_i^{\mathsf T},\\
 M_{\beta_0}&=b\bar w_0\sum_iP_i^{\mathsf T},\\
 (d_0,d_1,d_2)&=(0,3\delta,-3\delta)\log\rho+O(1).
 \end{aligned}
$$
The three terms have respective orders
$(3T)^{-1}(\rho,\rho^{1+3\delta},\rho^{1-3\delta})$
in the norm induced by $h_{0,T}$.
Differentiation bounds $|\bar\partial_ES_0|$ by
$C_I\rho^{-3\delta}\|\tau\|_Y$.  The Higgs term satisfies
$$
 |[\Phi_T,S_0]|\leq C_I\rho^{1-6\delta}\|\tau\|_Y
 \leq C_I\rho^{-3\delta}\|\tau\|_Y,\qquad 0<\rho\leq1,
$$
since $1-3\delta>0$.  Hence
\begin{equation}\label{eq:vertical-local-bounds}
 |S_0|\leq C_IT^{-1}\rho^{1-3\delta}\|\tau\|_Y,
 \qquad
 |V_0|\leq C_I\rho^{-3\delta}\|\tau\|_Y.
\end{equation}
Equivariance of multiplication and the local metric makes $S_0$
descend, and its trace
vanishes by the character calculations above.  Its norm tends to zero
since $1-3\delta>0$.  The corresponding local integrals are bounded by
$$
 C_IT^{-4}\|\tau\|_Y^2\int_0^1\rho^{3-6\delta}\,d\rho,
 \qquad
 C_IT^{-2}\|\tau\|_Y^2\int_0^1\rho^{1-6\delta}\,d\rho,
$$
for the potential and $D_T''S_0$, respectively.  Both are
finite throughout the chamber.  At $p=1,\infty$, the model is decoupled,
so $S_p=M_{\beta_p}$ and $V_p=\mathbf F_\tau$.  The cutoff criterion
now gives $t_\tau\in\mathcal Q$.

Since $V_0=D_T''S_0$ and $(D_T'')^{*,T}V_0=0$, Green's formula on
$A_{r,\varrho_0}$ gives
\begin{equation}\label{eq:vertical-disk-energy}
 \langle V_0,V_0\rangle_{A_{r,\varrho_0}}
 =\mathfrak d_{\varrho_0}^T(S_0,V_0)
  -\mathfrak d_r^T(S_0,V_0).
\end{equation}
For $\rho\leq1$, \eqref{eq:vertical-local-bounds} gives
$$
 |\mathfrak d_r^T(S_0,V_0)|
 \leq C_IT^{-6\delta}r^{2-6\delta}\|\tau\|_Y^2
 \leq C_Ir^{2-6\sup I}\|\tau\|_Y^2.
$$
For $\rho\geq1$, Toda decay gives $|S_0|\leq C_Ir\|\tau\|_Y$ and
$|V_0|\leq C_I\|\tau\|_Y$, so the boundary term is bounded by
$C_Ir^2\|\tau\|_Y^2$.  The inner boundary in
\eqref{eq:vertical-disk-energy} thus vanishes uniformly, and the fixed
outer boundary is bounded by $C_I\|\tau\|_Y^2$.  This proves the disk
norm bound.  The spectral $L^2$ identity on the exterior and Toda decay
on the fixed transition annuli give the global estimate.  Closedness
and trace-freeness follow from $(D_T'')^2=0$ and the corresponding
properties of $\mathbf F_{\tau'}$ and $S_p$.
\end{proof}

\subsection{Comparison of the pairings}

Under the spectral splitting the semi-flat Hermitian form is
$$
 \langle(\nu_1,\tau_1),(\nu_2,\tau_2)\rangle_{\rm sf}
 =2\ii\int_Y\nu_1\wedge\overline{\nu_2}
  -2\ii\int_Y\tau_1\wedge\overline{\tau_2}.
$$
Indeed, the direct-image trace sums the sheets and integration on the
base contributes the reciprocal covering degree.  Thus the decoupled
pairing is the spectral Hodge pairing.  The Gauss--Manin horizontal
and vertical tangent subspaces are orthogonal; see
\cite[Proposition~2.14 and Theorem~2.15]{Fredrickson-1810.01554} and
\cite[Proposition~4.6]{Mochizuki23}.  In particular,
$\|(\nu,\tau)\|_{\rm sf}^2=\|\nu\|_Y^2+\|\tau\|_Y^2$.

\begin{prop}
\label{prop:spectral-pairings}
For $x_j=(\nu_j,\tau_j)$ and $T$ sufficiently large,
\begin{equation}\label{eq:approximate-spectral-pairings}
 \begin{gathered}
 \left|
 \left\langle\mathbf H'(\nu_1,T)+\mathbf V'(\tau_1,T),
               \mathbf H'(\nu_2,T)+\mathbf V'(\tau_2,T)\right\rangle_{\rm app}
 -\langle x_1,x_2\rangle_{\rm sf}
 \right|\\
 \leq Ce^{-cT}\|x_1\|_{\rm sf}\|x_2\|_{\rm sf}.
 \end{gathered}
\end{equation}
\end{prop}

\begin{proof}
At $p=1,\infty$ the representatives equal their decoupled limits.  On
the fixed transition annuli, Toda decay bounds all pairing errors by
$CT^Ne^{-cT}\|x_1\|_{\rm sf}\|x_2\|_{\rm sf}$.  We may absorb $T^N$
by decreasing $c$.  It remains to compare the three pairings on the
Toda disk.

The horizontal pairing is Proposition~\ref{lem:horizontal-representative}.
For the vertical pairing, apply \eqref{eq:vertical-disk-energy}
to the model and its decoupled limit.  Both inner boundary terms vanish
by the estimates in Proposition~\ref{lem:vertical-representative}.
At the fixed outer circle, the difference of the potentials, metrics
and their required derivatives is exponentially small.  Adding the
transition and exterior estimates compares the squared norms.
Hermitian polarization then gives
$$
 \left|
 \langle\mathbf V'(\tau_1,T),\mathbf V'(\tau_2,T)\rangle_{\rm app}
 +2\ii\int_Y\tau_1\wedge\overline{\tau_2}
 \right|
 \leq Ce^{-cT}\|\tau_1\|_Y\|\tau_2\|_Y.
$$
The mixed pairing vanishes pointwise.  In the cyclic frame,
$a_{\mathbf H'}$ is diagonal and $\varphi_{\mathbf H'}$ lies in
$\operatorname{span}\{P_i\}$, whereas $a_{\mathbf V'}$ lies in
$\operatorname{span}\{P_i^{\mathsf T}\}$ and
$\varphi_{\mathbf V'}$ is diagonal.  These properties follow from
the displayed formulas for $s_\xi$ and $S_0$ and persist after
multiplication by the cutoffs.  Since $h_T^{\rm app}$ is diagonal
in this frame, both component pairings vanish on the Toda disk and
its transition annulus.  On the exterior and at $p=1,\infty$, the
horizontal representative has only a Higgs component and the
vertical representative has only a connection component.  Thus
\begin{equation}\label{eq:approximate-mixed-pairing}
 \langle\mathbf H'(\nu,T),\mathbf V'(\tau,T)\rangle_{\rm app}=0.
\end{equation}
Combining the horizontal and vertical estimates proves the result.
\end{proof}

We project the approximate representatives onto the
$L^2$-orthogonal with respect to $h_T$ complement of $D_T''\mathcal Q$ in the space
of deformation pairs.  The section
$(D_T'')^{*,\mathrm{app}}(\mathbf H'+\mathbf V')$ is supported on
the transition annuli.  Coercivity of $Q_{\rm app}$, followed by
comparison with $Q_{\rm ex}$, controls the required correction.

\begin{lem}
\label{lem:complex-coulomb-coercivity}
For $T$ sufficiently large, $Q_{\rm app}(r,r)\geq c_I\|r\|_{L^2}^2$
on $\mathcal Q$, and
$$
 \left|\langle\mathbf H'(\nu,T)+\mathbf V'(\tau,T),D_T''r
              \rangle_{\rm app}\right|
 \leq Ce^{-cT}\|(\nu,\tau)\|_{\rm sf}Q_{\rm app}(r,r)^{1/2}.
$$
\end{lem}

\begin{proof}
Write $\mathfrak q=\mathfrak q_{\lambda,q,\delta}$ for the
real Friedrichs form defined above, and set
$\mathfrak q^{\C}(u+\ii v,u+\ii v)=\mathfrak q(u,u)+\mathfrak q(v,v)$
for self-adjoint $u,v$.  Let $A$ be the Chern connection of
$h_T^{\rm app}$, and take adjoints with respect to this metric.
For a compactly supported section $r=u+\ii v$, the quadratic identity
in Lemma~\ref{lem:fixed-linearization} gives
$$
 \mathfrak q^{\C}(r,r)
 =\|d_Ar\|_{L^2}^2+\|[\Phi_T,r]\|_{L^2}^2
                        +\|[\Phi_T^\dagger,r]\|_{L^2}^2.
$$
Here the norms are the zero- and one-form norms from
Section~\ref{sec:global-analysis}.  The K\"ahler identities and
integration by parts give
$$
 2\|\bar\partial_Er\|_{L^2}^2
 =\|d_Ar\|_{L^2}^2-\langle[\ii\star_cF_A,r],r\rangle_{L^2},
$$
while the commutator identity gives
$$
 \|[\Phi_T^\dagger,r]\|_{L^2}^2-\|[\Phi_T,r]\|_{L^2}^2
 =\langle[\ii\star_c[\Phi_T,\Phi_T^\dagger],r],r\rangle_{L^2}.
$$
Since $Q_{\rm app}(r,r)=2\|\bar\partial_Er\|_{L^2}^2
+2\|[\Phi_T,r]\|_{L^2}^2$, these identities imply
\begin{equation}\label{eq:complex-coercivity-identity}
 Q_{\rm app}(r,r)=\mathfrak q^{\C}(r,r)
   -\langle[R_{\lambda,q,\delta},r],r\rangle_{L^2}.
\end{equation}
By Propositions~\ref{prop:uniform-coercivity} and
\ref{prop:residual-decay},
$$
 |\langle[R_{\lambda,q,\delta},r],r\rangle_{L^2}|
 \leq2\|R_{\lambda,q,\delta}\|_{L^\infty,\operatorname{op}}\|r\|_{L^2}^2
 \leq\tfrac12\mathfrak q^{\C}(r,r)
$$
for large $T$.  Consequently the identity on compactly supported complex sections
extends to an identification
$\operatorname{Dom}(\mathfrak q^{\C})=\mathcal Q$ with equivalent
norms.  Thus $Q_{\rm app}$ has the asserted uniform coercivity on
$\mathcal Q$.

Put $\mathbf a=\mathbf H'(\nu,T)+\mathbf V'(\tau,T)$.
On each inner disk the local Coulomb equation holds.  Outside the
cutoff disks, $h_T^{\rm app}=h_q^{\rm dec}$ and the representatives
are $\mathbf F_\nu$ and $\mathbf F_\tau$.  The flat eigenline
connections and harmonicity of $\nu$ and $\tau$ give
$(D_T'')^{*,\mathrm{app}}\mathbf F_\nu=
 (D_T'')^{*,\mathrm{app}}\mathbf F_\tau=0$ there.
Hence $f=(D_T'')^{*,\mathrm{app}}\mathbf a$ is supported on the
fixed transition annuli.  The representative formulas and Toda decay give
$\|f\|_{L^2}\leq CT^Ne^{-cT}\|(\nu,\tau)\|_{\rm sf}$.  For compactly
supported $r$, integration by parts therefore gives
$$
 |\langle\mathbf a,D_T''r\rangle_{\rm app}|
 =|\langle f,r\rangle_{L^2}|
 \leq CT^Ne^{-cT}\|(\nu,\tau)\|_{\rm sf}\|r\|_{L^2}.
$$
Coercivity bounds the last factor by $CQ_{\rm app}(r,r)^{1/2}$.
Density extends the estimate to $\mathcal Q$, and decreasing $c$
absorbs the polynomial factor.
\end{proof}

\begin{prop}
\label{prop:exact-harmonic-projection}
For $x=(\nu,\tau)$ there is a unique $s_x\in\mathcal Q$ such that
\begin{equation}\label{eq:harmonic-projection}
 Q_{\rm ex}(s_x,r)
 =\langle\mathbf H'(\nu,T)+\mathbf V'(\tau,T),D_T''r\rangle_{\rm ex}
\end{equation}
for every $r\in\mathcal Q$.  It satisfies $\|D_T''s_x\|_{\rm ex}\leq Ce^{-cT}\|x\|_{\rm sf}$.
The pair
$$
 \mathbf Z_x=\mathbf H'(\nu,T)+\mathbf V'(\tau,T)-D_T''s_x
$$
is the $h_T$-harmonic representative of $x$, and
$$
 \left|\langle\mathbf Z_{x_1},\mathbf Z_{x_2}\rangle_{\rm ex}
       -\langle x_1,x_2\rangle_{\rm sf}\right|
 \leq Ce^{-cT}\|x_1\|_{\rm sf}\|x_2\|_{\rm sf}.
$$
\end{prop}

\begin{proof}
Proposition~\ref{prop:correction} gives
$$
 h_T(u,v)=h_T^{\rm app}(e^{\gamma_T}u,e^{\gamma_T}v),
 \qquad \|\gamma_T\|_{L^\infty}\leq Ce^{-cT}.
$$
For an endomorphism, the pointwise comparison follows by conjugation
with $e^{\gamma_T}$.  The operator $e^{\gamma_T}$ and its inverse have
norm at most $e^{\|\gamma_T\|_{L^\infty}}$.  Applying this to the two components of
$D_T''r$ gives
$$
 e^{-4\|\gamma_T\|_{L^\infty}}Q_{\rm app}(r,r)
 \leq Q_{\rm ex}(r,r)
 \leq e^{4\|\gamma_T\|_{L^\infty}}Q_{\rm app}(r,r).
$$
The same comparison holds for the zero-form norm.  Thus the exact and
approximate form norms have the same completion $\mathcal Q$.  If
$\mathbf a_x=\mathbf H'(\nu,T)+\mathbf V'(\tau,T)$, then
$$
 |\langle\mathbf a_x,D_T''r\rangle_{\rm ex}
       -\langle\mathbf a_x,D_T''r\rangle_{\rm app}|
 \leq Ce^{-cT}\|x\|_{\rm sf}Q_{\rm app}(r,r)^{1/2}.
$$
Lemma~\ref{lem:complex-coulomb-coercivity} gives the same bound for the
approximate pairing and transfers coercivity to $Q_{\rm ex}$.  The Riesz
representation theorem on the Hilbert space $(\mathcal Q,Q_{\rm ex})$
gives $s_x$.  Testing its defining
equation with $r=s_x$ yields
$$
 Q_{\rm ex}(s_x,s_x)
 \leq Ce^{-cT}\|x\|_{\rm sf}Q_{\rm ex}(s_x,s_x)^{1/2},
$$
which is the stated norm bound.
The pair $\mathbf Z_x$ is closed and orthogonal to $D_T''\mathcal Q$.

To identify the deformation class represented by $\mathbf Z_x$,
fix the parameter and put $F=\mathbf F_\nu+\mathbf F_{\tau'}$ and
$s_{\rm app}=-\chi_0s_\xi+t_\tau$.  Then
$s_{\rm app},s_x\in\mathcal Q$ and
$\mathbf Z_x=F+D_T''(s_{\rm app}-s_x)$, where $F$ is a deformation
of the strongly parabolic Higgs bundle: its horizontal part differentiates
$\Phi_{\lambda,q}$ with $q$ fixed, and its vertical part was obtained from the
localized spectral family.  For the fixed weights, there is a
compatible exponent $1<p<2$, chosen sufficiently close to one as in
\cite[Definition~2.5]{Konno1993}.  The harmonic metric supplied by
\cite[Theorem~4.5]{Konno1993} has the form $H=Ke^\ell$, where $K$
is the reference metric.  With $g=e^{-\ell/2}$, the pair
$(g^{-1}\bar\partial_Eg,g^{-1}\Phi_Tg)$ is harmonic for $K$.
Uniqueness in Proposition~\ref{prop:correction} identifies $H=h_T$,
so $g$ is an isometry from the reference bundle with metric $K$ to
the bundle with metric $h_T$.  Conjugation by $g$ transports the
Hodge decomposition in \cite[Proposition~3.6(2)]{Konno1993} to
$h_T$ and gives
$$
 F=\mathbf Z_{\rm geo}+D_T''u.
$$
By \cite[Section~3.3]{Konno1993}, the harmonic term
$\mathbf Z_{\rm geo}$ represents the geometric tangent class of $F$.

In the adapted unitary frames for the reference metric $K$, the
Sobolev spaces in this decomposition use ordinary $W^{j,p}$ norms
on weight-preserving entries and weighted Sobolev norms on the other
entries.  For $1<p<2$, \cite[Lemma~2.2 and Section~2.2]{Konno1993}
identifies the weighted second-order space on a disk with
$\{v\in W^{2,p}:v(0)=0\}$ and the weighted first-order space with
$W^{1,p}$.  Thus $u$ has componentwise $W^{2,p}$ regularity, and
$D_T''u$ has componentwise $W^{1,p}$ regularity in these frames.  Sobolev embedding
gives $|u|_{h_T}\in L^\infty$ and
$D_T''u\in L^{2p/(2-p)}\subset L^2$.  The isometric transport and the
conformal invariance of the one-form $L^2$ norm preserve these
properties.  The logarithmic cutoff criterion therefore gives
$u\in\mathcal Q$.  Both $\mathbf Z_x$ and $\mathbf Z_{\rm geo}$ are
orthogonal to $D_T''\mathcal Q$, the latter by density from compact
supports.  Taking $w=u+s_{\rm app}-s_x\in\mathcal Q$, we obtain
$$
 \mathbf Z_x-\mathbf Z_{\rm geo}=D_T''w,
 \qquad
 \|\mathbf Z_x-\mathbf Z_{\rm geo}\|_{\rm ex}^2
 =\langle\mathbf Z_x-\mathbf Z_{\rm geo},D_T''w\rangle_{\rm ex}=0.
$$
At each fixed parameter, this identifies the geometric tangent class.
The uniform estimates follow from coercivity of $Q_{\rm ex}$ and
the preceding bound for $D_T''s_x$.

Finally, orthogonality gives
\begin{equation}\label{eq:projected-pairing}
 \langle\mathbf Z_{x_1},\mathbf Z_{x_2}\rangle_{\rm ex}
 =\langle\mathbf a_{x_1},\mathbf a_{x_2}\rangle_{\rm ex}
  -Q_{\rm ex}(s_{x_1},s_{x_2}).
\end{equation}
The last term is $O(e^{-2cT})\|x_1\|_{\rm sf}\|x_2\|_{\rm sf}$.
Pointwise metric equivalence changes the first pairing to its
approximate value with an $O(e^{-cT})$ error.  Proposition~\ref{prop:spectral-pairings}
then proves the claimed comparison.
\end{proof}

\begin{thm}
\label{thm:metric-comparison}
For every $I\Subset(0,1/3)$ there are $C_I,c_I,T_I>0$ such that
\begin{equation}\label{eq:hitchin-semiflat-comparison}
 \sup_{\substack{\delta\in I\\(\phi,q)\in S^1\times\Pic^0(Y)}}
 |g_{L^2,\delta}-g_{{\rm sf},\delta}|_{g_{{\rm sf},\delta}}
 \leq C_Ie^{-c_IT},\qquad T\geq T_I.
\end{equation}
\end{thm}

\begin{proof}
The Toda decay estimates, the residual bound, coercivity, and the
correction estimate are uniform on $I\times S^1\times\Pic^0(Y)$.  The exponents controlling
the inner boundary terms are bounded below by a positive number
depending only on $I$.
The spectral constructions use a finite cover of the compact parameter
set, so their constants are uniform as well.
Proposition~\ref{prop:exact-harmonic-projection} consequently gives
$$
 |g_{L^2,\delta}(x_1,x_2)-g_{{\rm sf},\delta}(x_1,x_2)|
 \leq C_Ie^{-c_IT}\|x_1\|_{{\rm sf},\delta}\|x_2\|_{{\rm sf},\delta}.
$$
Evaluation in a semi-flat orthonormal frame, with a dimensional constant
absorbed into $C_I$, proves \eqref{eq:hitchin-semiflat-comparison}.
\end{proof}

\section{ALG asymptotics and compactification}
\label{sec:alg}

The metric comparison identifies the leading term on the cubic
cover $\C_\lambda^\times\times\Pic^0(Y)$ of the regular moduli space.
We first identify the elliptic compactification, then compute the
periods of the spectral differential, the integral lattice of the
Jacobian, and the covering action.  These determine the metric and
holomorphic symplectic form in the standard ALG coordinates.
We then prove completeness and obtain decay of all derivatives by
local Ricci-flat K\"ahler estimates.  The complex surface and flat end
model are independent of the parabolic weight within the chamber.

\subsection{Symplectic form and elliptic compactification}

Let $S=\mu_3\text{-Hilb}(T^*Y)$ denote the component containing
reduced free three-point orbits for the cotangent lift of $\sigma$.
Its symplectic form $\omega_S$ is normalized as in
\cite[Definition~2.4]{Jia26}.

\begin{prop}
\label{prop:compactification}
The surface $S$ is isomorphic to $\calM_\delta$ for every $0<\delta<1/3$.  The relative
Fourier--Mukai isomorphism followed by tensoring with $L_\delta$ pulls
$\Omega_\delta^{\rm Serre}$ back to $\omega_S$.
\end{prop}

\begin{proof}
\cite[Theorem~4.1]{Gro12} applies to the orbifold
$[Y/\mu_3]$ with full flags, weights $(0,1/3,2/3)$ at $1,\infty$,
and final weight in $(2/3,1)$ at the distinguished point $0$.
Our original weights satisfy these conditions.  The underlying degree
is $-3$, as determined by degree zero for the unperturbed orbifold
weights.  Proposition~\ref{prop:chamber} then gives the isomorphism
with $\calM_\delta$ and preserves the deformation pairing.  The full parabolic deformation complex splits as
\begin{equation}\label{eq:trace-deformation-splitting}
 \begin{aligned}
 &[\ParEnd E\xrightarrow{\ad\Phi}\SParEnd E\otimes K_X(D)]\\
 &\qquad\simeq
 [\ParEnd_0E\xrightarrow{\ad\Phi}\SParEnd_0E\otimes K_X(D)]
 \oplus[\mathcal O_X\xrightarrow0K_X].
 \end{aligned}
\end{equation}
Indeed, the projections are $A\mapsto A-\frac13\tr(A)\Id$ and
$B\mapsto B-\frac13\tr(B)\Id$; the trace of a strongly parabolic
Higgs field is holomorphic.  Since
$H^1(\Pp^1,\mathcal O)\oplus H^0(\Pp^1,K_X)=0$, inclusion of the
trace-free complex identifies the tangent spaces and their Serre
pairings.  We apply the calculation in \cite[Lemma~3.2 and the proof
of Proposition~3.3]{Jia26} to this Fourier--Mukai isomorphism with the
weights fixed above.  For each fixed Higgs bundle with full flags, these deformation sheaves
and their trace pairing depend on the flags and the Higgs field.
They are unchanged when the numerical weights vary within the chamber.  The Fourier--Mukai construction gives an orbifold Higgs
bundle on $[Y/\mu_3]$.  On its cover $Y$, a free cotangent orbit
transforms into the direct sum of its three rank-one transforms.
At a nonzero cotangent coordinate, their eigenforms are distinct.
Each off-diagonal deformation differential is therefore multiplication
by a nowhere-vanishing holomorphic one-form, so its complex is acyclic.
The three diagonal complexes give the rank-one deformation spaces.
The rank-one Higgs moduli space is $T^*\Pic^0(Y)$, identified with
$T^*Y$ by the cotangent lift of the Abel--Jacobi isomorphism with
origin $r_0$.  After pullback to the orbifold cover $Y$, the
differential of the Fourier--Mukai isomorphism sends a tangent vector
$u$ at one point of the orbit to the three diagonal variations
$((\sigma^j)_*u)_{j=0}^2$.
Writing $\omega_{T^*Y}$ for the canonical holomorphic symplectic
form on $T^*Y$, the pairings on the three rank-one factors and the
integration factor $1/3$ give
\begin{equation}\label{eq:symplectic-orbit-pairing}
 \frac13\sum_{j=0}^2
 \omega_{T^*Y}\bigl((\sigma^j)_*u,(\sigma^j)_*v\bigr)
 =\omega_{T^*Y}(u,v).
\end{equation}
Here the cotangent lift of $\sigma$ preserves $\omega_{T^*Y}$.
This is the form descending to the free-orbit locus of $S$.
The two regular forms thus agree on a dense open subset and hence on
the smooth surface $S$.
\end{proof}

The differential $\alpha$ trivializes $T^*Y$.  Since
$$
 f_\lambda^*\theta=\lambda\alpha,
 \qquad \alpha^3=-\pi^*q_1,
$$
the quotient coordinate is $t=\lambda^3$, in agreement with
$\det(\eta-\Phi)=\eta^3+tq_1$; compare
\cite[Theorem~3.8]{Jia26}.

The cotangent action of $\mu_3$ extends to $Y\times\Pp^1$.
In the minimal resolution of its quotient, let $D_0,D_\infty$ denote
the strict transforms of the images of $Y\times\{0\}$ and
$Y\times\{\infty\}$.  Over the three singularities at zero, let
$E_{1j}\cup E_{2j}$ be the exceptional chains, indexed so that
$E_{1j}$ meets $D_0$.  Let $F_j$ be the exceptional curve over the
$j$-th singularity at infinity, for $1\leq j\leq3$.
All these curves are smooth and rational.  The exceptional curves
satisfy $E_{1j}^2=E_{2j}^2=-2$ and $F_j^2=-3$, as follows from the
local quotient types $\frac13(1,2)$ at zero and $\frac13(1,1)$ at
infinity; see \cite[Lemma~6.9]{Huang26}.

\begin{thm}
\label{thm:elliptic-compactification}
The minimal resolution above is a smooth projective rational
compactification of $S$, equipped with an elliptic fibration whose
only singular fibers are
\begin{equation}\label{eq:singular-fibers}
 3D_0+\sum_{j=1}^3(2E_{1j}+E_{2j})
 \quad\text{and}\quad
 3D_\infty+\sum_{j=1}^3F_j.
\end{equation}
The first has Kodaira type $IV^*$.  The component $D_\infty$ is a
$(-1)$-curve; its contraction gives a relatively minimal elliptic
surface with a fiber of type $IV$ at infinity.
\end{thm}

\begin{proof}
The total ramification of $Y\to\Pp^1$ over $0,1,\infty$ identifies
$[Y/\mu_3]$ with $\Pp^1_{3,3,3}$.  The orbifold Hilbert scheme
compactification and its elliptic morphism are therefore those of
\cite[Theorems~1.5--1.6 and Proposition~6.3]{Huang26}.  The two fibers
and the contraction at infinity are given by
\cite[Lemma~6.10]{Huang26}.  The relation $\lambda^3=t$ identifies their
base with the Hitchin base while preserving the marked values zero and
infinity.  In particular, the stated central fiber is the nilpotent
fiber of the Hitchin morphism.
\end{proof}

The compactification in this theorem retains the boundary component
$D_\infty$.  Its contraction is supported entirely in the boundary,
so it leaves the open surface $S$ unchanged.  On the relatively minimal
model, the boundary is the type-$IV$ fiber; the central type-$IV^*$
fiber remains unchanged and has dual intersection graph of type
$\widetilde E_6$.
This specifies the algebraic configuration entering the type $E_6$
in Theorem~\ref{thm:main}.  The analytic ALG parameters will be fixed
by the metric and period calculation below.

\subsection{The flat metric on the end}

\begin{prop}
\label{prop:flat-end}
For every $R>0$, the quotient isomorphism $\Psi_\delta$ restricts to
a biholomorphism
$$
 \Psi_\delta:
 \bigl(\{|\lambda|>R\}\times\Pic^0(Y)\bigr)/\mu_3
 \xrightarrow{\sim}h^{-1}\{|t|>R^3\},
$$
where $(\lambda,q)\mapsto(\omega\lambda,\omega^{-1}q)$ acts freely.
The sublevel sets of $|h|$ form a cofinal compact exhaustion, and
$\calM_\delta$ has one end.
\end{prop}

\begin{proof}
Propositions~\ref{prop:spectral} and~\ref{prop:spectral-quotient} give
the quotient description.  The action is free because $\lambda\ne0$.
Properness of the Hitchin map \cite[Theorem~2.10]{Gro12} makes its
closed sublevel sets compact, and every compact subset has bounded
image.  Their complements have the displayed connected quotient
parametrization.
\end{proof}

The proposition above identifies the full end of the moduli space.  On its cubic cover
the elliptic curve and its integral homology are fixed, so the angular
coordinates on $\Pic^0(Y)$ can be chosen independently of $\lambda$.  The
finite quotient records their monodromy over the $t$-plane.

We choose a symplectic basis $A,B$ on $Y$ as follows.  On $(0,1)$ let
$\gamma_k$ be the lift from $0$ to $1$ on the sheet
$\zeta_k(x)=\omega^k[x(1-x)]^{1/3}$, and set
$A=\gamma_0-\gamma_1$ and $B=\gamma_1-\gamma_2$.
The graph $\pi^{-1}([0,1])$ has two vertices, three edges, and a disk
as its complement in $Y$, so $A,B$ form an integral homology basis.
With the complex orientation, endpoint perturbations give
$A\cdot B=1$.  Since $\sigma\gamma_k=\gamma_{k+1}$,
the induced action on homology is
$$
 \sigma_*A=B,\qquad \sigma_*B=-A-B.
$$
On a parameter chart in $\Pic^0(Y)$, choose smooth trivializations
of the degree-zero lines $P_q$ and write their unitary flat connections
as $d+\ii\chi_q$, with $\chi_q$ real harmonic.  Changes of gauge
shift its cohomology class by $2\pi H^1(Y;\Z)$.  This identifies the
underlying real torus with $H^1(Y;\R)/2\pi H^1(Y;\Z)$.  We
use the period-$2\pi$ angles
$$
 \theta_A=\int_A\chi_q,\qquad
 \theta_B=\int_B\chi_q\pmod{2\pi}.
$$
The special coordinate $a$ and the flat coordinates are
\begin{equation}\label{eq:flat-end-coordinates}
 a=-\frac{3e^{-\ii\pi/6}\Gamma(1/3)^3}{2\pi}\lambda,
 \qquad
 u=2\ii\sqrt{\Im\omega}\,a,
 \qquad
 v=\frac{\theta_B-\omega\theta_A}{\sqrt{\Im\omega}}.
\end{equation}
The class $v$ is understood modulo the lattice in the next statement.

\begin{prop}
\label{prop:flat-coordinates}
On the cubic cover, the coordinates \eqref{eq:flat-end-coordinates}
give
\begin{equation}\label{eq:flat-model-identities}
 \begin{aligned}
 \widetilde\Psi_\delta^*g_{{\rm sf},\delta}
 &=g_0=|du|^2+|dv|^2,\\
 \widetilde\Psi_\delta^*\Omega_\delta&=du\wedge dv,\qquad
 \widetilde\Psi_\delta^*\Omega_\delta^{\rm Serre}
 =\frac{du\wedge dv}{4\pi\ii}.
 \end{aligned}
\end{equation}
The fiber lattice is
$\Lambda=\frac{2\sqrt2\pi}{3^{1/4}}(\Z+\omega\Z)$ and the deck
action is $(u,v)\mapsto(\omega u,\omega^{-1}v)$.  Thus
$\Psi_\delta^*g_{{\rm sf},\delta}=g_{\mathcal C}$ for the flat
ALG model with parameters
$$
 (\beta,\tau,L)=
 \left(\frac13,\omega,\frac{2\sqrt2\pi}{3^{1/4}}\right).
$$
\end{prop}

\begin{proof}
On the $k$-th sheet,
$\alpha=-\omega^k x^{-2/3}(1-x)^{-2/3}\,dx$.  Both endpoint
singularities are integrable.  Euler's beta integral and reflection
formula give
$$
 \begin{aligned}
 \int_0^1x^{-2/3}(1-x)^{-2/3}\,dx
 &=\frac{\Gamma(1/3)^2}{\Gamma(2/3)}
 =\frac{\sqrt3\,\Gamma(1/3)^3}{2\pi},\\
 \Gamma(1/3)\Gamma(2/3)&=\frac{2\pi}{\sqrt3},
 \qquad 1-\omega=\sqrt3e^{-\ii\pi/6}.
 \end{aligned}
$$
Subtracting the sheet integrals along the chosen cycles yields
\begin{equation}\label{eq:spectral-periods}
 \begin{aligned}
 \int_A\alpha
 &=-(1-\omega)\frac{\Gamma(1/3)^2}{\Gamma(2/3)}
 =-\frac{3e^{-\ii\pi/6}\Gamma(1/3)^3}{2\pi},\\
 \int_B\alpha&=\omega\int_A\alpha.
 \end{aligned}
\end{equation}
Consequently, $a=\int_Af_\lambda^*\theta$, and
$$
 \alpha_0=-\frac{2\pi e^{\ii\pi/6}}{3\Gamma(1/3)^3}\alpha,
 \qquad
 \int_A\alpha_0=1,\quad \int_B\alpha_0=\omega,
 \qquad
 \ii\int_Y\alpha_0\wedge\overline{\alpha_0}=2\Im\omega.
$$
For a real harmonic variation
$\dot\chi=c\alpha_0+\bar c\,\overline{\alpha_0}$ with periods
$p_A,p_B$, the period equations are
$$
 p_A=c+\bar c,\qquad
 p_B=c\omega+\bar c\,\bar\omega.
$$
Solving them gives
\begin{equation}\label{eq:vertical-dolbeault-periods}
 c=\frac{p_B-\bar\omega p_A}{\omega-\bar\omega},
 \qquad
 \tau=(\ii\dot\chi)^{0,1}
 =-\frac{p_B-\omega p_A}{2\Im\omega}\overline{\alpha_0}.
\end{equation}
Large unitary gauge transformations shift $(\theta_A,\theta_B)$ by
$2\pi\Z^2$.  Since $\Pic^0(\Pp^1)=0$, the norm map on $\Pic^0(Y)$ is
trivial.  The corresponding shifts of $v$ are
$$
 \Delta v=\frac{2\pi}{\sqrt{\Im\omega}}(n-\omega m),
 \qquad (m,n)\in\Z^2.
$$
Thus
\begin{equation}\label{eq:flat-torus-lattice}
 \Lambda=\frac{2\pi}{\sqrt{\Im\omega}}(\Z+\omega\Z)
 =\frac{2\sqrt2\pi}{3^{1/4}}(\Z+\omega\Z).
\end{equation}
Formula~\eqref{eq:vertical-dolbeault-periods} is complex linear in
$p_B-\omega p_A$, so $v$ is a holomorphic coordinate on the universal
cover of $\Pic^0(Y)$ and $dv$ descends as its invariant holomorphic one-form.

The spectral pairings of Proposition~\ref{prop:spectral-pairings} and
the Riemann bilinear identity give
\begin{equation}\label{eq:semiflat-period-expression}
 \begin{aligned}
 \widetilde\Psi_\delta^*g_{{\rm sf},\delta}
 &=4\Im\omega\,|da|^2
   +\frac{|d\theta_B-\omega d\theta_A|^2}{\Im\omega}\\
 &=|du|^2+|dv|^2.
 \end{aligned}
\end{equation}
For the horizontal term, the variation is
$d\lambda\,\alpha=da\,\alpha_0$ and its squared norm is
$2\ii|da|^2\int_Y\alpha_0\wedge\overline{\alpha_0}$.
For the vertical term, $*\alpha_0=-\ii\alpha_0$ and
$*\overline{\alpha_0}=\ii\overline{\alpha_0}$, so
$$
 \int_Y\dot\chi\wedge*\dot\chi
 =2\ii|c|^2\int_Y\alpha_0\wedge\overline{\alpha_0}
 =4\Im\omega\,|c|^2
 =\frac{|p_B-\omega p_A|^2}{\Im\omega}.
$$
The horizontal and vertical spaces are orthogonal.  These formulas use
the unnormalized matrix trace: integration contributes $1/3$ and the
three spectral summands contribute $3$, as in
Proposition~\ref{prop:compactification}.

Finally, with deformation pairs ordered as $(a,\varphi)$,
\eqref{eq:holomorphic-symplectic-pairing} gives the normalization
of the holomorphic symplectic form induced by trace and Serre duality:
\begin{equation}\label{eq:serre-normalization}
 \begin{aligned}
 \Omega_\delta^{\rm Serre}(\mathbf x_1,\mathbf x_2)
 &=\frac1{2\pi\ii}\int_X\tr(
 \varphi_1\wedge a_2-\varphi_2\wedge a_1),\\
 \Omega_\delta&=4\pi\ii\,\Omega_\delta^{\rm Serre}.
 \end{aligned}
\end{equation}
For a horizontal tangent vector with special-coordinate variation
$\dot a$ and a vertical tangent vector $\dot q$ represented by $\tau$,
\eqref{eq:vertical-dolbeault-periods} gives
$$
 \int_Y(\dot a\,\alpha_0)\wedge\tau
 =\ii\dot a\,(p_B-\omega p_A),
 \qquad p_A=d\theta_A(\dot q),\quad p_B=d\theta_B(\dot q).
$$
The Abel--Jacobi map $\operatorname{AJ}(s)=\mathcal O_Y(s-r_0)$
and Serre duality give
\begin{equation}\label{eq:abel-jacobi-normalization}
 \begin{aligned}
 \frac1{2\pi\ii}\int_Y\alpha_0\wedge\tau
 &=\frac{p_B-\omega p_A}{2\pi},\qquad
 \tau=(\ii\dot\chi)^{0,1},\\
 \operatorname{AJ}^*\left(
 \frac{d\theta_B-\omega d\theta_A}{2\pi}\right)&=\alpha_0;
 \end{aligned}
\end{equation}
compare \cite[Section~3.1, equation~(3.1)]{Jia26}.
In these coordinates, the canonical holomorphic symplectic form
on $T^*\Pic^0(Y)$ is
$\frac1{2\pi}\,da\wedge(d\theta_B-\omega d\theta_A)$.
The normalized integral is one third of the sum of the three equal
diagonal contributions, as in \eqref{eq:symplectic-orbit-pairing}.
The constant diagonal terms pair trivially with
strictly flag-lowering residues, and the remaining boundary terms
vanish by Proposition~\ref{lem:vertical-representative}.  Therefore
$$
 \widetilde\Psi_\delta^*\Omega_\delta
 =2\ii\,da\wedge(d\theta_B-\omega d\theta_A)=du\wedge dv.
$$
To check the action on the angular coordinates, the map
$q\mapsto\omega^{-1}q$ is induced by $\sigma^*$ on the spectral
line.  The homology formulas therefore give
$$
 (\theta_A,\theta_B)\longmapsto
 (\theta_B,-\theta_A-\theta_B),
 \qquad
 \theta_B-\omega\theta_A\longmapsto
 \omega^{-1}(\theta_B-\omega\theta_A).
$$
Since $a$ is a constant multiple of $\lambda$, the full action is
$(u,v)\mapsto(\omega u,\omega^{-1}v)$.  It preserves both the metric
and $du\wedge dv$, and agrees with \cite[Definition~2.5]{CV24}.
\end{proof}

In particular, the flat model is obtained from
$\{|u|>R_0\}\times(\C/\Lambda)$ with the product Euclidean metric
by this order-three action.  In polar coordinates $u=Re^{\ii\vartheta}$
it has local expression
$$
 g_{\mathcal C}=dR^2+R^2d\vartheta^2+|dv|^2,
 \qquad
 (\vartheta,v)\sim(\vartheta+2\pi/3,\omega^{-1}v).
$$
The base cone angle is $2\pi/3$, hence $\beta=1/3$, while the fixed
elliptic modulus is $\tau=\omega$.  The map induced by
$t=\lambda^3$ is the Hitchin fibration on the four-dimensional
quotient.  It identifies the base quotient
$\{|\lambda|>R\}/\mu_3$ with $\{|t|>R^3\}$, with $R$ related to
$R_0$ by the scaling \eqref{eq:flat-end-coordinates}.

The elementary lattice invariants are
\begin{equation}\label{eq:flat-torus-geometry}
 \Area(\C/\Lambda)=L^2\Im\omega=4\pi^2,
 \qquad
 \inj(\C/\Lambda)=\frac L2=\frac{\sqrt2\pi}{3^{1/4}},
\end{equation}
since $|m+n\omega|^2=m^2-mn+n^2\geq1$ for
$(m,n)\in\Z^2\setminus\{(0,0)\}$.  The radial coordinate is
\begin{equation}\label{eq:alg-radial-scale}
 R_{\rm ALG}=|u|
 =\frac{3^{5/4}\Gamma(1/3)^3}{\sqrt2\pi}|\lambda|
 =\frac{3^{5/4}\Gamma(1/3)^3}{\sqrt2\pi}|t|^{1/3}.
\end{equation}
All these coordinates, tensors, and constants are independent of
$\delta$.

\subsection{Completeness and decay of derivatives}

Write $g_\delta=\widetilde\Psi_\delta^*g_{L^2,\delta}$ on the cover.
The map $\widetilde\Psi_\delta$ is holomorphic, so $g_0$ and $g_\delta$ are
K\"ahler for the same Dolbeault complex structure.
The differences between the largest and smallest normalized weights
are $2/3$ and $2/3+\delta<1$;
the adapted harmonic metrics and parabolic hyperk\"ahler quotient are
those of \cite{Simpson1990,Konno1993}.  The finite scalar stabilizer
has trivial action on the deformation space.  Thus $g_\delta$ is a
smooth hyperk\"ahler metric and, in particular, Ricci-flat K\"ahler.

\begin{prop}
\label{prop:completeness-volume}
The metric $g_{L^2,\delta}$ is complete and has quadratic volume
growth: for some $p\in\calM_\delta$ and positive $c,C,\rho_0$,
$$
 c\rho^2\leq\Vol_{g_{L^2,\delta}}B(p,\rho)\leq C\rho^2
 \qquad(\rho\geq\rho_0).
$$
\end{prop}

\begin{proof}
Theorem~\ref{thm:metric-comparison} and
\eqref{eq:alg-radial-scale} give
$\frac12g_0\leq g_\delta\leq\frac32g_0$ outside a compact set.
This comparison descends to the quotient end via $\Psi_\delta$.
Writing $R=R_{\rm ALG}$ there, every curve satisfies
$L_{g_{L^2,\delta}}(\gamma)\geq2^{-1/2}|\Delta R|$.
Proposition~\ref{prop:flat-end} therefore implies that every divergent
curve has infinite length, which proves completeness.
Choose $p$ on the compact side of a level $R=R_0$.  For a point
$x$ in the end, radial paths and the last crossing of this level give
$$
 2^{-1/2}(R(x)-R_0)
 \leq d_{g_{L^2,\delta}}(p,x)
 \leq C+(3/2)^{1/2}(R(x)-R_0).
$$
The bounded constant in the upper bound includes paths from $p$ to
points of the compact level.  The lower bound applies to the final
part of any path after its last crossing of that level.  For the upper
bound, first reach the corresponding point at $R_0$, then use a radial
path with fixed angular and torus coordinates on the cover.
The flat torus area and quotient angle give
\begin{equation}\label{eq:model-volume}
 \Vol_{g_{\mathcal C}}\{R_0<R_{\rm ALG}<S\}
 =\frac{4\pi^3}{3}(S^2-R_0^2).
\end{equation}
On the same quotient end,
$$
 \frac14\,dV_{g_{\mathcal C}}
 \leq dV_{\Psi_\delta^*g_{L^2,\delta}}
 \leq\frac94\,dV_{g_{\mathcal C}}.
$$
The distance bounds compare geodesic balls with radial sublevel sets.
Their quadratic volumes, together with the finite volume of the compact part,
prove the claim.
\end{proof}

To apply local Ricci-flat estimates, fix
$0<r<\inj(\C/\Lambda)=\sqrt2\pi/3^{1/4}$.  Outside a fixed compact set,
every ball of radius $r$ for the flat quotient lifts holomorphically
and isometrically to a Euclidean ball.  Indeed, nonzero lattice
translations have length at least $2\inj(\C/\Lambda)>2r$, while a deck
transformation with rotation $\omega^k$, $k=1,2$, displaces a point by
at least $|1-\omega^k||u|=\sqrt3|u|$.  More precisely, all such
transformations have the form
$$
 (u,v)\longmapsto(\omega^ku,\omega^{-k}v+\ell),
 \qquad \ell\in\Lambda.
$$
Choose a threshold $R_*>R_0+r$ with
$\sqrt3(R_*-r)>2r$.  Every point of a ball centered at $|u|>R_*$
then remains in the end and has displacement greater than $2r$
under each transformation with $k\ne0$.  Two points of that ball
cannot be identified, since their distance is less than $2r$.
This proves that the ball is disjoint from all its nontrivial
translates.  The radius and threshold are independent of $\delta$.
On these fixed holomorphic balls, the zeroth-order comparison and
local elliptic estimates give bounds for higher derivatives.

\begin{thm}
\label{thm:derivative-decay}
Let $I\Subset(0,1/3)$.  On the three-sheeted cyclic cover of an end
region chosen uniformly for $\delta\in I$, there is a constant
$c_I>0$ such that, for every integer $k\geq0$,
\begin{equation}\label{eq:all-derivative-decay}
 \sup_{\delta\in I}
 |(\nabla^{g_0})^k(g_\delta-g_0)|_{g_0}
 \leq C_{I,k}e^{-c_IR_{\rm ALG}}.
\end{equation}
The constants $C_{I,k}>0$ are independent of $\delta$.
The estimates are invariant under the lattice and $\mu_3$ actions
and descend to the standard ALG model.
\end{thm}

\begin{proof}
Fix $0<r<\sqrt2\pi/3^{1/4}$.  The flat metric, its lattice, and the
radius $r$ are independent of $\delta$.
The preceding construction gives the same lifted Euclidean
holomorphic balls throughout $I$.  Theorem~\ref{thm:metric-comparison}
and \eqref{eq:alg-radial-scale} give constants $C,a>0$ such that, for
every $x$ in the end,
$$
 \sup_{\delta\in I}|g_\delta-g_0|_{g_0}(x)
 \leq Ce^{-aR_{\rm ALG}(x)}.
$$
Since $R_{\rm ALG}=|u|$ is one-Lipschitz for $g_0$, this estimate
holds on each fixed ball after changing $C$.  Increasing the end radius
uniformly for $\delta\in I$ gives
$$
 \frac12g_0\leq g_\delta\leq\frac32g_0
 \qquad(\delta\in I).
$$

Lift a ball to $\C^2$, translate its center to zero, and rescale by the
fixed radius.  Every $g_\delta$ is Ricci-flat K\"ahler for the same
complex structure and has the same ellipticity ratio.  The local
estimate \cite[Section~2.1, Proposition~2.1]{HT20} consequently gives,
for every integer $m\geq1$,
$$
 \sup_{\delta\in I}
 \|g_\delta-g_0\|_{C^m(B_{3r/4})}\leq C_{I,m}.
$$
For a scalar component $f$, the interpolation inequality on these balls is
$$
 \|D^jf\|_{L^\infty(B_{r/2})}
 \leq
 C\|f\|_{L^\infty(B_{3r/4})}^{1-j/m}
 \|f\|_{C^m(B_{3r/4})}^{j/m}
 +C\|f\|_{L^\infty(B_{3r/4})}.
$$
For each integer $k\geq1$, take $j=k$ and $m=2k$.  The zeroth-order exponential
estimate and the uniform $C^{2k}$ bound yield
$$
 \sup_{\delta\in I}|D^k(g_\delta-g_0)|
 \leq C_{I,k}e^{-(a/2)R_{\rm ALG}}.
$$
The same rate also bounds the case $k=0$ after enlarging its constant.
Thus $c_I=a/2$ works for every $k$.  Since $g_0$ is flat,
coordinate derivatives are its covariant derivatives.
The constants in the local regularity estimate depend on the chosen
derivative order, ball radius, and ellipticity ratio.  These quantities
are fixed throughout $I$.

The lattice translations and
$(u,v)\mapsto(\omega u,\omega^{-1}v)$ preserve $g_0$,
$g_\delta$, and $R_{\rm ALG}$.  The estimates therefore descend
with unchanged constants.
\end{proof}

The same exponential estimate holds for the curvature and all its
covariant derivatives.  In the following formulas, $*$ denotes a
finite universal linear combination of tensor contractions.  In flat
coordinates,
$$
 \begin{aligned}
 \Rm(g_\delta)
 &=g_\delta^{-1}*D^2g_\delta
 +g_\delta^{-1}*g_\delta^{-1}*Dg_\delta*Dg_\delta,\\
 \nabla^{g_\delta}-\nabla^{g_0}
 &=g_\delta^{-1}*Dg_\delta.
 \end{aligned}
$$
Uniform ellipticity and Theorem~\ref{thm:derivative-decay} therefore give
$$
 \sup_{\delta\in I}\bigl(
 |(\nabla^{g_0})^\ell\Rm(g_\delta)|_{g_0}
 +|(\nabla^{g_\delta})^\ell\Rm(g_\delta)|_{g_\delta}\bigr)
 \leq C_{I,\ell}e^{-c_IR_{\rm ALG}}
 \qquad(\ell\in\Z_{\geq0}).
$$
Since $Dg_0=0$, every differentiated term contains a derivative of
positive order of $g_\delta-g_0$, so the rate $c_I$ is unchanged.

\begin{proof}[Proof of Theorem~\ref{thm:main}]
Propositions~\ref{prop:flat-end} and~\ref{prop:flat-coordinates} give
the end identification, the semi-flat metric and holomorphic
symplectic form in the normalized coordinates, and the stated
ALG parameters.
Proposition~\ref{prop:completeness-volume} gives completeness and
quadratic volume growth.  Theorem~\ref{thm:derivative-decay} gives the
claimed decay, uniformly on compact subintervals of the weight
chamber.  Since both metrics are K\"ahler for the Dolbeault complex
structure, the same decay holds for their first K\"ahler forms; the equality of their holomorphic symplectic
forms identifies the other two K\"ahler forms exactly.

The quotient is the standard model
$\mathcal C_{1/3,\omega,2\sqrt2\pi/3^{1/4}}$ of
\cite[Definitions~1.1 and~2.5, Remark~2.7 and Table~2.1]{CV24}.
For each $k$, exponential decay implies
$e^{-c_IR_{\rm ALG}}=O(R_{\rm ALG}^{-2-k})$.
Thus the metric and K\"ahler-form estimates satisfy the order-$2$
ALG conditions in the stated model, and completeness gives an ALG
gravitational instanton.  Its curvature is square-integrable:
the end volume density is bounded by a constant times
$R_{\rm ALG}\,dR_{\rm ALG}$ after integration over the compact
angular and torus variables, and
$$
 \int_{R_0}^{\infty}
 R_{\rm ALG}e^{-2c_IR_{\rm ALG}}\,dR_{\rm ALG}<\infty.
$$
Theorem~\ref{thm:elliptic-compactification} identifies the central
fiber with the affine $E_6$ configuration.  It contains rational
curves.  A complete flat K\"ahler surface has holomorphic universal
cover $\C^2$, which admits only constant holomorphic maps from
$\Pp^1$.  Thus the metric is nonflat.
\end{proof}

\end{document}